\documentclass[10pt,reqno,draft]{amsart}
\usepackage[latin9]{inputenc}
\usepackage{color}
\usepackage{textcomp}
\usepackage{amstext}
\usepackage{amsthm}
\usepackage{amssymb}
\usepackage{esint}

\makeatletter
\numberwithin{equation}{section}
\numberwithin{figure}{section}
\theoremstyle{plain}
\newtheorem{thm}{\protect\theoremname}[section]
\theoremstyle{plain}
\newtheorem{assumption}[thm]{\protect\assumptionname}
\theoremstyle{remark}
\newtheorem{rem}[thm]{\protect\remarkname}
\theoremstyle{plain}
\newtheorem{lem}[thm]{\protect\lemmaname}
\theoremstyle{plain}
\newtheorem{prop}[thm]{\protect\propositionname}

\usepackage{color}
\usepackage{latexsym}
\usepackage{bm}
\usepackage{graphicx}
\usepackage{wrapfig}
\usepackage{fancybox}

     \def\section{\@startsection{section}{1}%
     \z@{.7\linespacing\@plus\linespacing}{.5\linespacing}%
     {\bfseries
     \centering
     }}
     \def\@secnumfont{\bfseries}

\newcommand{\Rd}{\mathbb{R}^{d}}

\newcommand{\BN}{\mathbb{N}}

\newcommand{\BR}{\mathbb{R}}

\newcommand{\bfE}{\mathbf{E}}

  \providecommand{\lemmaname}{Lemma}
  \providecommand{\propositionname}{Proposition}
  \providecommand{\remarkname}{Remark}
\providecommand{\theoremname}{Theorem}
\providecommand{\assumptionname}{Assumption}

  \providecommand{\lemmaname}{Lemma}
  \providecommand{\propositionname}{Proposition}
  \providecommand{\remarkname}{Remark}
\providecommand{\theoremname}{Theorem}

  \providecommand{\lemmaname}{Lemma}
  \providecommand{\propositionname}{Proposition}
  \providecommand{\remarkname}{Remark}
\providecommand{\theoremname}{Theorem}

  \providecommand{\lemmaname}{Lemma}
  \providecommand{\propositionname}{Proposition}
  \providecommand{\remarkname}{Remark}
\providecommand{\theoremname}{Theorem}

\makeatother

\providecommand{\assumptionname}{Assumption}
\providecommand{\lemmaname}{Lemma}
\providecommand{\propositionname}{Proposition}
\providecommand{\remarkname}{Remark}
\providecommand{\theoremname}{Theorem}

\begin{document}
\title[Uniqueness in law for stable-like processes of variable order]{Uniqueness in law for stable-like processes of variable order}
\author[P. Jin]{Peng Jin}

\address{Peng Jin: Department of Mathematics, Shantou University, Shantou, Guangdong 515063, China}

\email{pjin@stu.edu.cn}



\begin{abstract} Let $d\ge1$. Consider a stable-like operator of
variable order
\begin{align*}
\mathcal{A}f(x) & =\int_{\mathbb{R}^{d}\backslash\{0\}}\left[f(x+h)-f(x)-\mathbf{1}_{\{|h|\le1\}}h\cdot\nabla f(x)\right]n(x,h)|h|^{-d-\alpha(x)}\mathrm{d}h,
\end{align*}
where $0<\inf_{x}\alpha(x)\le\sup_{x}\alpha(x)<2$ and $n(x,h)$ satisfies
\[
n(x,h)=n(x,-h),\quad0<\kappa_{1}\le n(x,h)\le\kappa_{2},\quad\forall x,h\in\mathbb{R}^{d},
\]
with $\kappa_{1}$ and $\kappa_{2}$ being some positive constants.
Under some further mild conditions on the functions $n(x,h)$ and
$\alpha(x)$, we show the uniqueness of solutions to the martingale
problem for $\mathcal{A}$. \end{abstract}

\maketitle

\section{Introduction}

Consider the non-local operator 
\begin{equation}
\mathcal{A}f(x)=\int_{\mathbb{R}^{d}\backslash\{0\}}\left[f(x+h)-f(x)-\mathbf{1}_{\{|h|\le1\}}h\cdot\nabla f(x)\right]\frac{n(x,h)}{|h|^{d+\alpha(x)}}\mathrm{d}h,\label{eq: defi A}
\end{equation}
where $n(x,h)$ is bounded above and below by positive constants and
$0<\inf_{x}\alpha(x)\le\sup_{x}\alpha(x)<2$. Due to the fact that
the jump kernel $n(x,h)/|h|^{d+\alpha(x)}$ is comparable to that
of an isotropic stable process of order $\alpha(x)$, with $\alpha(x)$
depending on $x$, the operator $\mathcal{A}$ is called a stable-like
operator of variable order. Operators of the form (\ref{eq: defi A})
were already investigated, for instance, in  \cite{MR2095633,MR2180302,MR2244602,MR958291,tang2007uniqueness}.
However, many problems related to $\mathcal{A}$ have not been fully
understood. The variable order nature of $\mathcal{A}$, in contrast
to constant order stable-like operators, brings us many difficulties. 

In \cite{MR2095633,MR2180302} Bass and Kassmann proved the Harnack
inequalities and regularity of harmonic functions with respect to
$\mathcal{A}$. There, as one part of the standing assumption, the
existence of a strong Markov process associated with $\mathcal{A}$
was assumed. In fact, their results were proved via probabilistic
method where the strong Markov property played an important role.
Later, Silvestre \cite{MR2244602} obtained Hölder regularity of harmonic
functions with respect to more general non-local operators, and his
approach was purely analytical. 

The existence of a strong Markov process associated with $\mathcal{A}$
is closely related to the corresponding martingale problem (see below
for the definition). In the case where $\alpha(x)\equiv\alpha$ is
constant, the well-posedness of the martingale problem for $\mathcal{A}$
(possibly with lower order perturbations) was proved in \cite{MR1246036,MR2583323,MR2508568,MR3145767,MR3201992,chen2016uniqueness}
under various assumptions; in particular, Mikulevi\v{c}ius and Pragarauskas
\cite{MR3145767} obtained the well-posedness by requiring the Hölder
continuity of $x\mapsto n(x,h)$. Recently, by establishing some estimate
of Krylov\textquoteright s type, Chen and Zhang \cite{chen2016uniqueness}
extended the result of \cite{MR3145767} to much more general (constant
order) stable-like operators with possibly singular jump measures
which are comparable to those of nondegenerate $\alpha$-stable processes. 

The martingale problem for $\mathcal{A}$ becomes more delicate when
$\alpha(x)$ is allowed to change with $x$. For sufficiently smooth
functions $n(x,h)$ and $\alpha(x)$, the operator $\mathcal{A}$
and its martingale problem can be studied using the classical theory
of pseudo-differential operators, see \cite{MR1243995,MR1378858,MR1809340}.
However, with coefficients that are not smooth, this approach fails
to work. In the general case, the solvability of the martingale problem
for $\mathcal{A}$ is actually not difficult to obtain by the weak
convergence argument, and the reader is referred to \cite{MR0433614,MR958291,tang2007uniqueness}
for some sufficient conditions for existence. In contrast, the uniqueness
problem is more difficult. For one spatial dimension a condition for
uniqueness was given by Bass \cite{MR958291}, where some perturbation
method was used. With a similar idea, Tang \cite{tang2007uniqueness}
considered the more general multidimensional case and provided also
a sufficient condition for uniqueness; however, the condition \cite[Assumption 2.2(a)]{tang2007uniqueness}
there (see also Remark \ref{rem:(i)-According-to} below), which is
necessary to make the approach to work, seems a bit restrictive to
rule out some interesting cases. 

We would like to mention that if one considers solutions of stochastic
differential equations driven by stable processes, it is also possible
to obtain Markov processes that are of variable order nature. For
example, consider the following system of SDEs 
\begin{equation}
\begin{cases}
dX_{t}^{i}=\sum_{j=1}^{d}A_{ij}(X_{t-})dZ_{t}^{j}, & i\in\{1,\ldots,d\},\\
\ X_{0}=x_{0}\in\mathbb{R}^{d},
\end{cases}\label{eq: system of sdes}
\end{equation}
where $A=(A_{ij}):\mathbb{R}^{d}\to\mathbb{R}^{d\times d}$ is measureable
and $Z_{t}^{1},\ldots,Z_{t}^{d}$ are independent one-dimensional
symmetric stable processes with stability indices $\alpha_{1},\ldots,\alpha_{d}\in(0,2)$.
In \cite{MR2222382}, Bass and Chen showed that if $\alpha_{1}=\ldots=\alpha_{d}$
and the matrix $A(x)$ is continuous in $x$ and non-degenerate, then
the system \eqref{eq: system of sdes} has a unique weak solution.
Recently, Chaker \cite{Chaker19} studied the variable order case
and showed that if $A(x)$ is diagonal, non-degenerate and bounded
continuous, then weak uniqueness for \eqref{eq: system of sdes} also
holds. However, weak uniqueness for the general variable order case
of \eqref{eq: system of sdes} remains unsolved. 

The aim of this paper is to study the uniqueness for the martingale
problem associated with the operator $\mathcal{A}$ defined in \eqref{eq: defi A},
without assuming too strong regularity conditions on its coefficients.
Our standing assumption on the functions $n(x,h)$ and $\alpha(x)$
reads as follows. 
\begin{assumption}
\label{assu: main}Suppose

\emph{(a) for }$x,h\in\Rd$, $n(x,h)=n(x,-h)$ and $0<\kappa_{1}\le n(x,h)\le\kappa_{2}<\infty$,
where $\kappa_{1},\kappa_{2}$ are constants;

\emph{(b)} $\int_{0}^{1}r^{-1}\psi(r)\mathrm{d}r<\infty$, where $\psi(r):=\sup_{h\in\Rd,|x-y|\le r}|n(x,h)-n(y,h)|$; 

\emph{(c) for $x\in\Rd$,} $0<\underline{\alpha}\le\alpha(x)\le\overline{\alpha}<2$,
where $\underline{\alpha},\overline{\alpha}$ are constants;

\emph{(d) $\beta(r)=o(|\ln r|^{-1})$ as $r\to0$ and} $\int_{0}^{1}r^{-1}|\ln r|\beta(r)\mathrm{d}r<\infty$,
where $\beta(r):=\sup_{|x-y|\le r}|\alpha(x)-\alpha(y)|$. 
\end{assumption}

\begin{rem}
\label{rem:(i)-According-to} According to Assumption \ref{assu: main}(b),
$n(x,h)$ is Dini continuous in $x$. Note that the condition in \cite[Assumption 2.2(a)]{tang2007uniqueness}
is very different from ours and requires the existence of a Dini continuous
function $\xi(x)$ such that $|n(x,h)-\xi(x)|\leq c_{1}(1\wedge|h|^{\epsilon})$
for all $x,h\in\Rd$, where $c_{1,}\epsilon>0$ are some constants.
In fact, the essential idea of \cite{tang2007uniqueness} is to view
the jump kernel $n(x,h)|h|^{-d-\alpha(x)}$ as a perturbation of the
kernel $\xi(x)|h|^{-d-\alpha(x)}$.
\end{rem}

Under Assumption \ref{assu: main}, the existence for the martingale
problem associated with $\mathcal{A}$ is guaranteed, due to \cite[Theorem 2.2]{MR0433614}.
Our main result for uniqueness is the following. 

For the sake of completeness we first recall the definition of the
martingale problem for $\mathcal{A}$. Let $D=D\big([0,\infty);\Rd\big)$,
the set of paths in $\mathbb{R}^{d}$ that are right continuous with
left limits, be endowed with the Skorokhod topology. Set $X_{t}(\omega)=\omega(t)$
for $\omega\in D$ and let $\mathcal{D}=\sigma(X_{t}:0\le t<\infty)$
and $\mathcal{F}_{t}:=\sigma(X_{r}:0\le r\le t)$. A probability measure
$\mathbf{P}$ on $(D,\mathcal{D})$ is called a solution to the \emph{martingale
problem} for $\mathcal{A}$ starting from $x\in\Rd$, if $\mathbf{P}(X_{0}=x)=1$
and under the measure $\mathbf{P}$, 
\[
f(X_{t})-\int_{0}^{t}\mathcal{A}f(X_{u})\mathrm{d}u,\ t\ge0,
\]
 is an $\left(\mathcal{F}_{t}\right)$-martingale for all $f\in C_{b}^{2}(\mathbb{R}^{d})$. 
\begin{thm}
\label{thm: main}Let $\mathcal{A}$ be as in \emph{(\ref{eq: defi A})},
and suppose Assumption \emph{\ref{assu: main}} holds. Then for each
$x\in\Rd$, the martingale problem for the operator $\mathcal{A}$
starting from $x$ has at most one solution.
\end{thm}

In Theorem \ref{thm: main} our assumption on the functions $n(x,h)$
and $\alpha(x)$ is very mild. As a result, the weak uniqueness for
a large class of variable order stable-like processes now follows.
It's also worth noting that, even in the special case that $\alpha(x)$
is constant, Theorem \ref{thm: main} provides some new result for
uniqueness, since our assumption that $x\mapsto n(x,h)$ is Dini continuous
improves the Hölder continuity condition required in \cite{MR3145767}. 

To prove Theorem \ref{thm: main}, we use the technique introduced
in \cite{MR2642351}, where the uniqueness for martingale problem
was discussed in the context of elliptic diffusions. The core of this
technique is to approximate the semigroup of $\mathcal{A}$ by a mixture
of semigroups corresponding to constant coefficient operators $\mathcal{A}^{y}$
given by 
\[
\mathcal{A}^{y}f(x):=\int_{\mathbb{R}^{d}\backslash\{0\}}\left[f(x+h)-f(x)-\mathbf{1}_{\{|h|\le1\}}h\cdot\nabla f(x)\right]\frac{n(y,h)}{|h|^{d+\alpha(y)}}\mathrm{d}h.
\]
The method in \cite{MR2642351} is essentially a perturbation technique
which has its root in the parametrix method for the construction of
fundamental solutions of parabolic equations. The same idea was later
used in \cite{MR2802040,MR3573301} to obtain weak uniqueness of solutions
to some degenerate SDEs. Note that the approach in \cite{MR958291,tang2007uniqueness}
are similar to \cite{MR2642351}, with the difference that the perturbation
is carried out on the resolvent of $\mathcal{A}$. 

We now give a few remarks on some possible extensions of Theorem \ref{thm: main}. 
\begin{rem}
(1) Instead of $\mathcal{A}$, one can also consider the more general
operator $\tilde{\mathcal{A}}f(x):=\mathcal{A}f(x)+b(x)\cdot\nabla f(x)$,
where $b:\mathbb{R}^{d}\to\mathbb{R}^{d}$ is bounded and Dini continuous
(in the sense of Assumption \ref{assu: main}(b)). If Assumption \ref{assu: main}
holds and, in addition, $\inf_{x}\alpha(x)>1$, then we can combine
our methods and those of \cite{kulik2015weak} to show uniqueness
of the martingale problem for $\tilde{\mathcal{A}}$. 

(2) For simplicity, in this paper we have assumed the symmetry of
$n(x,h)$ in $h$. However, due to the recent works of \cite{MR3806688}
and \cite{jin2017heat}, it is not difficult to extend Theorem \ref{thm: main}
to non-symmetric $n(x,h)$ under the additional assumption that either
$\inf_{x}\alpha(x)>1$ or $\sup_{x}\alpha(x)<1$. 
\end{rem}

Recently, there has been a lot of works that exploit the parametrix
method to study the heat kernel of jump processes, see e.g. \cite{chen2015heat,Kim2017,MR3806688,jin2017heat,MR3765882,MR3877025,MR3912204,chen2019heat}.
It is worthwhile to mention that the above list is, by far, not complete.
For the variable order operator $\mathcal{A}$ as in \eqref{eq: defi A},
its heat kernel has been constructed and estimated in \cite{chen2019heat}.
Therein, the authors assumed slightly stronger conditions than we
did in Assumption \ref{assu: main}; more precisely, they assumed
additionally that $n(x,h),\alpha(x)$ are both Hölder continuous in
$x$, and that $\inf_{x}\alpha(x)$, $\sup_{x}\alpha(x)$ satisfy
an inequality so that the oscillation of the function $\alpha(x)$
can not be too large (see also Section 4 below for a similar condition
we will assume). It is an interesting question whether the results
of \cite{chen2019heat} can be extended to the case where $x\mapsto n(x,h)$
and $x\mapsto\alpha(x)$ merely satisfy some continuity condition
of Dini's type.\textcolor{red}{{} }

Let us eventually point out the fact that the term ``stable-like''
process is now broadly used in the literature, so that in a different
context it might means a process that differs from what we consider
here. For other types of stable-like processes (either symmetric or
non-symmetric), the reader is referred to \cite{MR1274897,MR1262797,MR1880349,MR2008600,MR2091550,MR2357678}. 

The rest of the paper is organized as follows. After a section on
preliminaries, where we collect some basic facts on stable-like Lévy
processes, in Section 3 we define the parametrix and derive some estimates
for it. In Sections 4 we prove a special case of Theorem \ref{thm: main},
namely, under the additional assumption that $\overline{\alpha}<2\underline{\alpha}$.
In Section 5 we remove this constraint and prove Theorem \ref{thm: main}
in its general form. 

\section{Preliminaries}

\subsection{Notation}

Here we give a few remarks on our notation. The letter $c$ with subscripts
will denote positive finite constants whose exact value is unimportant.
We write $C(d,\lambda,...)$ for a positive finite constant $C$ that
depends only on the parameters $d,\lambda,....$ For a function $f$
on $\Rd$, we will use $f(x\pm z)$ to denote $f(x+z)+f(x-z)$. If
$f$ is bounded, we write $\|f\|:=\sup_{x\in\Rd}|f(x)|$.  

\subsection{Convolution inequalities}

Throughout this section, let $[\alpha_{1},\alpha_{2}]$ be a compact
subinterval of the interval $(0,2)$. For $\beta,\gamma\in\BR$ and
$\alpha\in(0,2)$, we write 
\[
\varrho_{\alpha}^{\beta,\gamma}(t,x):=t^{\gamma/\alpha}(|x|^{\beta}\wedge1)(t^{1/\alpha}+|x|)^{-d-\alpha},\quad(t,x)\in(0,\infty)\times\Rd.
\]

\begin{lem}
\emph{\label{lem: conv ineq of chen}} There exists $C=C(d,\alpha_{1},\alpha_{2})>0$
such that for all $\alpha\in[\alpha_{1},\alpha_{2}]$ and ${\color{red}{\color{black}t>0}}$,
\begin{equation}
\int_{\Rd}\varrho_{\alpha}^{0,\alpha}(t,x)\mathrm{d}x\le C\quad\mbox{\mbox{and}\quad\ensuremath{\int_{\Rd}|\ln|x||\varrho_{\alpha}^{0,\alpha}(t,x)\mathrm{d}x\le C\left(1+|\ln t|\right).}}\label{esti1:rho}
\end{equation}
\end{lem}

\begin{proof}
We only prove the second inequality, since the first one is similar
and simpler. For $t=1$, we have 
\begin{align*}
\int_{\Rd}|\ln|x||\varrho_{\alpha}^{0,\alpha}(1,x)\mathrm{d}x & \le\int_{|x|\le1}|\ln|x||\mathrm{d}x+\int_{|x|>1}|x|^{-d-\alpha}|\ln|x||\mathrm{d}x\\
 & \le c_{1}+\int_{|x|>1}|x|^{-d-\alpha_{1}}|\ln|x||\mathrm{d}x\le c_{2},
\end{align*}
where $c_{2}=c_{2}(d,\alpha_{1})$ is a constant. For a general $t>0$,
by a change of variables $x':=t^{-1/\alpha}x$, we get
\begin{align*}
\int_{\Rd}|\ln|x||\varrho_{\alpha}^{0,\alpha}(t,x)\mathrm{d}x & \le\int_{\Rd}\left(\alpha^{-1}|\ln t|+|\ln|x||\right)\varrho_{\alpha}^{0,\alpha}(1,x)\mathrm{d}x\\
 & \le\alpha_{1}^{-1}|\ln t|\int_{\Rd}\varrho_{\alpha}^{0,\alpha}(1,x)\mathrm{d}x+c_{2}\le c_{3}\left(1+|\ln t|\right).
\end{align*}
 
\end{proof}
Later on, we need to compute convolutions of kernels $\varrho_{\alpha}^{0,\alpha}$
and $\varrho_{\tilde{\alpha}}^{0,\tilde{\alpha}}$ with different
indices $\alpha$ and $\tilde{\alpha}.$ The following inequality
(\ref{eq: Lemma 2.2}) provides an estimate on convolutions of this
type, which is not very precise but adequate for our propose. We remark
that a similar inequality to (\ref{eq: Lemma 2.2}) was implicitly
used in the proof of \cite[Lemma 5.2]{MR1744782}. 
\begin{lem}
\label{lem 1: conv different alpha} There exists $C=C(d,\alpha_{1},\alpha_{2})>0$
such that for all $|w|>0$ and $0<\tau<t\le1$,
\begin{align}
 & \int_{\Rd}\varrho_{\tilde{\alpha}}^{0,\tilde{\alpha}}(t-\tau,w-\eta)\varrho_{\alpha}^{0,\alpha}(\tau,\eta)\mathrm{d}\eta\nonumber \\
 & \quad\le C\exp\{|\alpha-\tilde{\alpha}|\cdot|\ln|w||\}\cdot\Big\{\varrho_{\alpha}^{0,\alpha}(t,w)+\varrho_{\tilde{\alpha}}^{0,\tilde{\alpha}}(t,w)\Big\},\label{eq: Lemma 2.2}
\end{align}
uniformly for $\alpha,\tilde{\alpha}\in[\alpha_{1},\alpha_{2}]$.
\end{lem}

\begin{proof}
We follow the proof of \cite[Lemma 5.2]{MR1744782}. Without loss
of generality, assume $\tilde{\alpha}\le\alpha$. Let $I$ denote
the integral from the left-hand side of (\ref{eq: Lemma 2.2}). We
need to consider two cases. 

(i) Suppose that $t^{1/\tilde{\alpha}}\leq|w|$. Let $D_{1}:={\displaystyle \{\eta:|w-\eta|\geq|w|/2\}}$
and $D_{2}$ be its complement. We now write $I=I_{1}+I_{2}$, where
$I_{1}$ and $I_{2}$ denote the corresponding integrals on $D_{1}$
and $D_{2}$, respectively. In $D_{1},$ 
\[
\varrho_{\tilde{\alpha}}^{0,\tilde{\alpha}}(t-\tau,w-\eta)\le\frac{t-\tau}{|w-\eta|^{d+\tilde{\alpha}}}\le c_{1}\frac{t}{|w|^{d+\tilde{\alpha}}}\le c_{2}\varrho_{\tilde{\alpha}}^{0,\tilde{\alpha}}(t,w),
\]
and therefore, due to Lemma \ref{lem: conv ineq of chen}, $I_{1}\le c_{3}\varrho_{\tilde{\alpha}}^{0,\tilde{\alpha}}(t,w)$.
Next, in $D_{2},$ we have $|\eta|\ge|w|/2$ and thus
\[
\varrho_{\alpha}^{0,\alpha}(\tau,\eta)\le\frac{\tau}{|\eta|^{d+\alpha}}\le\frac{c_{4}t}{|w|^{d+\alpha}}=\frac{c_{4}t}{|w|^{d+\tilde{\alpha}}}|w|^{\tilde{\alpha}-\alpha}\le c_{5}\varrho_{\tilde{\alpha}}^{0,\tilde{\alpha}}(t,w)\exp\{|\alpha-\tilde{\alpha}|\cdot|\ln|w||\},
\]
which implies $I_{2}\le c_{6}\varrho_{\tilde{\alpha}}^{0,\tilde{\alpha}}(t,w)\exp\{|\alpha-\tilde{\alpha}|\cdot|\ln|w||\}$.

(ii) Let ${\color{red}{\normalcolor |w|\leq t^{1/\tilde{\alpha}}\le t^{1/\alpha}}}$.
If ${\displaystyle \tau\geq t/2}$, then 
\[
\varrho_{\alpha}^{0,\alpha}(\tau,\eta)\le\tau^{-d/\alpha}\le c_{7}t^{-d/\alpha}\le c_{8}\varrho_{\alpha}^{0,\alpha}(t,w),
\]
and if $t-{\displaystyle \tau\geq t/2}$, then 
\[
\varrho_{\tilde{\alpha}}^{0,\tilde{\alpha}}(t-\tau,w-\eta)\le(t-\tau)^{-d/\tilde{\alpha}}\le c_{9}t^{-d/\tilde{\alpha}}\le c_{10}\varrho_{\tilde{\alpha}}^{0,\tilde{\alpha}}(t,w).
\]
In both cases we obtain the desired estimate by Lemma \ref{lem: conv ineq of chen}.
This completes the proof.
\end{proof}
\begin{lem}
There exists $C=C(d,\alpha_{1},\alpha_{2})>0$ such that for all $0<|w|\le1$
and $0<\tau<t\le1$,
\begin{align}
 & \int_{\Rd}\mathbf{1}_{\{|w-\eta|\ge2\}}\ln\left(|w-\eta|\right)\varrho_{\tilde{\alpha}}^{0,\tilde{\alpha}}(t-\tau,w-\eta)\varrho_{\alpha}^{0,\alpha}(\tau,\eta)\mathrm{d}\eta\nonumber \\
 & \quad\le C(1+|\ln\tau|+|\ln(t-\tau)|)\exp\{|\alpha-\tilde{\alpha}|\cdot|\ln|w||\}\nonumber \\
 & \qquad\times\Big\{\varrho_{\alpha}^{0,\alpha}(t,w)+\varrho_{\tilde{\alpha}}^{0,\tilde{\alpha}}(t,w)\Big\},\label{eq: Lemma 2.4}
\end{align}
uniformly for $\alpha,\tilde{\alpha}\in[\alpha_{1},\alpha_{2}]$.
\end{lem}

\begin{proof}
Without loss of generality, assume $\tilde{\alpha}\le\alpha$. Let
$I$ denote the integral from the left-hand side of (\ref{eq: Lemma 2.4}).
Note that if $|w-\eta|\ge2$, then $|\eta|\ge1\ge|w|$ and thus
\begin{equation}
\mathbf{1}_{\{|w-\eta|\ge2\}}\ln|w-\eta|\le\mathbf{1}_{\{|w-\eta|\ge2\}}\ln\left(|\eta|+|w|\right)\le\ln\left(2|\eta|\right).\label{eq1: lemma 2.3}
\end{equation}

(1) Suppose $t^{1/\tilde{\alpha}}\leq|w|$. Define $D_{1}:={\displaystyle \{\eta:|w-\eta|\geq|w|/2\}}$
and $D_{2}$ as its complement. We now write $I=I_{1}+I_{2}$, where
$I_{1}$ and $I_{2}$ denote the corresponding integrals on $D_{1}$
and $D_{2}$, respectively. As shown in the proof of the preceding
lemma, in $D_{1},$ we have 
\[
\varrho_{\tilde{\alpha}}^{0,\tilde{\alpha}}(t-\tau,w-\eta)\le c_{1}\varrho_{\tilde{\alpha}}^{0,\tilde{\alpha}}(t,w),
\]
and, due to (\ref{eq1: lemma 2.3}) and Lemma \ref{lem: conv ineq of chen},
$I_{1}\le c_{2}(1+|\ln\tau|)\varrho_{\tilde{\alpha}}^{0,\tilde{\alpha}}(t,w)$.
Similarly, in $D_{2},$ we have 
\[
\varrho_{\alpha}^{0,\alpha}(\tau,\eta)\le c_{3}\varrho_{\tilde{\alpha}}^{0,\tilde{\alpha}}(t,w)\exp\{|\alpha-\tilde{\alpha}|\cdot|\ln|w||\}
\]
and $I_{2}\le c_{4}(1+|\ln\left(t-\tau\right)|)\varrho_{\tilde{\alpha}}^{0,\tilde{\alpha}}(t,w)\exp\{|\alpha-\tilde{\alpha}|\cdot|\ln|w||\}$.

(2) Let $|w|\leq t^{1/\tilde{\alpha}}$. In view of part (ii) of the
proof of Lemma \ref{lem 1: conv different alpha}, we obtain the same
estimate for the integral as in case (1). The proof of the lemma is
complete.
\end{proof}

\subsection{Density functions of stable-like Lévy processes}

In this section, as in the previous one, we assume that $[\alpha_{1},\alpha_{2}]$
is a compact subinterval of $(0,2)$. Moreover, let $\Lambda_{1},\Lambda_{2}$
be some fixed constants with $0<\Lambda_{1}<\Lambda_{2}<\infty$.

Consider a Lévy process $Z=(Z_{t})_{t\geq0}$ such that $Z_{0}=0$
a.s. and 
\begin{align*}
\bfE\big[e^{iZ_{t}\cdot u}\big] & =e^{-t\psi(u)},\quad u\in\Rd,\\
\psi(u) & =-\int_{\mathbb{R}^{d}\setminus\{0\}}\Big(e^{iu\cdot h}-1-\mathbf{1}_{\{|h|\le1\}}iu\cdot h\Big)K(h)\mathrm{d}h,
\end{align*}
where the function $K:\Rd\to\mathbb{R}$ satisfies 

\begin{equation}
K(h)=K(-h)\quad\mbox{and}\quad\frac{\Lambda_{1}}{|h|^{d+\alpha}}\le K(h)\le\frac{\Lambda_{2}}{|h|^{d+\alpha}},\quad h\in\mathbb{R}^{d},\label{condition1forK}
\end{equation}
for some $\alpha\in[\alpha_{1},\alpha_{2}]$.

In view of (\ref{condition1forK}), we call $Z$ a stable-like Lévy
process. Note that 
\begin{equation}
{\normalcolor \int_{\mathbb{R}^{d}\setminus\{0\}}\Big((1-\cos(u\cdot h))\Big)\frac{1}{|h|^{d+\alpha}}\mathrm{d}h=C_{\alpha}|u|^{\alpha}},\label{eq: psi-tilde psi}
\end{equation}
where 
\begin{equation}
C_{\alpha}:=\int_{\mathbb{R}^{d}\setminus\{0\}}(1-\cos(h_{1}))|h|^{-d-\alpha}\mathrm{d}h\label{eq: defi C_alpha}
\end{equation}
is a positive constant that depends continuously on $\alpha$. Since
$K(h)=K(-h)$, it holds that 
\begin{equation}
|e^{-t\psi(u)}|=\exp\left\{ -t\int_{\mathbb{R}^{d}\setminus\{0\}}(1-\cos(u\cdot h))K(h)\mathrm{d}h\right\} \overset{(\ref{condition1forK}),(\ref{eq: psi-tilde psi})}{\le}e^{-t\Lambda_{1}C_{\alpha}|u|^{\alpha}}.\label{MPAFLsect31-1-1}
\end{equation}

By (\ref{MPAFLsect31-1-1}), the law of $Z_{t}$ has a density $f_{t}^{(\alpha)}\in L^{1}(\mathbb{R}^{d})\cap C_{b}(\mathbb{R}^{d})$
that is given by 
\begin{equation}
f_{t}^{(\alpha)}(x)=\frac{1}{(2\pi)^{d}}\int_{\mathbb{R}^{d}}e^{-iu\cdot x}e^{-t\psi(u)}\mathrm{d}u,\quad x\in\mathbb{R}^{d},\ t>0.\label{MPAFLsect312-1}
\end{equation}

\begin{rem}
We have used the notation $f_{t}^{(\alpha)}$ to indicate its dependence
on $\alpha$ (see \eqref{condition1forK}). Here $\alpha$ is allowed
to vary between $\alpha_{1}$ and $\alpha_{2}$. On the other hand,
the constants $\alpha_{1},\alpha_{2},\Lambda_{1},\Lambda_{2}$ are
assumed to be fixed. 
\end{rem}

First, we have the following estimates on $f_{t}^{(\alpha)}$, which
is a special case of \cite[Proposition 3.2]{Kim2017}. 
\begin{lem}
\label{cor:For-each-k}For each \textcolor{black}{$k\in\mathbb{Z}_{+}$,
}there exists $C=C(d,\alpha_{1},\alpha_{2},\Lambda_{1},\Lambda_{2},k)>0$
such that \textcolor{black}{
\begin{equation}
|\nabla^{k}f_{t}^{(\alpha)}(x)|\le{\color{red}{\color{black}Ct^{1-k/\alpha}\left(t^{1/\alpha}+|x|\right)^{-d-\alpha}}},\quad x\in\Rd,\ t>0,\label{bound:gradientg_t-1}
\end{equation}
uniformly for $\alpha\in[\alpha_{1},\alpha_{2}]$.} 
\end{lem}

\begin{proof}
\textcolor{black}{Let $S=(S_{t})_{t\ge0}$ be a $d$-dimensional subordinate
Brownian motion via an independent subordinator with Laplace exponent
$\phi(\lambda)=\lambda^{\alpha/2}$. Set $\Phi(r)=r^{\alpha}$, $r>0$.
Then the characteristic exponent of $S$ is given by $\Phi(|u|)=|u|^{\alpha}$,
$u\in\Rd$. In view of (}\ref{eq: psi-tilde psi}\textcolor{black}{)
and (}\ref{eq: defi C_alpha}), the\textcolor{black}{{} Lévy measure
$\mu$ of $S$ has a density (with respect to the Lebesgue measure)
given by 
\[
j(h)=\frac{C_{\alpha}^{-1}}{|h|^{d+\alpha}},\quad h\in{\color{red}{\color{black}\Rd\setminus\{0\}}}.
\]
By (}\ref{condition1forK}\textcolor{black}{), we get 
\[
C_{\alpha}\Lambda_{1}j(h)\le K(h)\le C_{\alpha}\Lambda_{2}j(h),\quad h\in{\color{red}{\color{black}\Rd\setminus\{0\}}}.
\]
Set 
\[
\hat{\gamma}_{0}:=\max\left\{ \Lambda_{2}\sup_{\alpha\in[\alpha_{1},\alpha_{2}]}C_{\alpha},\ \Lambda_{1}^{-1}\sup_{\alpha\in[\alpha_{1},\alpha_{2}]}\left(C_{\alpha}\right)^{-1}\right\} .
\]
Then $\hat{\gamma}_{0}>0$ is a constant depending only on $d,\alpha_{1},\alpha_{2},\Lambda_{1},\Lambda_{2}$,
and we have 
\begin{equation}
\hat{\gamma}_{0}^{-1}j(h)\le K(h)\le\hat{\gamma}_{0}j(h),\quad h\in{\normalcolor {\color{red}{\color{black}\Rd\setminus\{0\}}}}.\label{eq:1, gamma hat}
\end{equation}
Note that 
\begin{equation}
\lambda^{\alpha_{1}}\Phi(r)\le\lambda^{\alpha}\Phi(r)=\Phi(\lambda r),\quad\lambda\geq1,r\geq1.\label{eq: lower scaling conditon}
\end{equation}
By }\eqref{eq:1, gamma hat} and \eqref{eq: lower scaling conditon},
we can\textcolor{black}{{} apply }\cite[Proposition 3.2]{Kim2017}\textcolor{black}{{}
 to find a constant $c_{1}=c_{1}(d,\alpha_{1},\alpha_{2},\Lambda_{1},\Lambda_{2},k)>0$
such that 
\begin{equation}
|\nabla^{k}f_{1}^{(\alpha)}(x)|\le{\color{red}{\color{black}c_{1}\left(1+|x|\right)^{-d-\alpha}}},\quad x\in\Rd.\label{eq: esti grad k f_1}
\end{equation}
}

Let $a>0$ and define $Y_{t}:=aZ_{a^{-\alpha}t}$, $t\ge0$. Then
$(Y_{t})$ is a pure-jump Lévy process with jump kernel $M(h):=a^{-d-\alpha}K(a^{-1}h)$,
$h\in\Rd$. Moreover, the function $M$ satisfies

\begin{equation}
M(h)=M(-h)\quad\mbox{and}\quad\frac{\Lambda_{1}}{|h|^{d+\alpha}}\le M(h)\le\frac{\Lambda_{2}}{|h|^{d+\alpha}},\quad h\in\mathbb{R}^{d}.\label{condition1forM}
\end{equation}
Therefore, $(Y_{t})$ is also a stable-like Lévy process. Let $\rho(x),$
$x\in\Rd,$ be the probability density of $Y_{1}$. By \eqref{condition1forM}
and (\ref{eq: esti grad k f_1}), we have 
\begin{equation}
|\nabla^{k}\rho(x)|\le{\color{red}{\color{black}c_{1}\left(1+|x|\right)^{-d-\alpha}}},\quad x\in\Rd.\label{eq: grad rho}
\end{equation}
We now choose $a$ such that $a^{-\alpha}=t$. Then $Y_{1}=t^{-1/\alpha}Z_{t}$
and 
\[
\rho(x)=t^{d/\alpha}f_{t}^{(\alpha)}(t^{1/\alpha}x),\quad x\in\Rd.
\]
So $\nabla^{k}\rho(x)=t^{\left(d+k\right)/\alpha}\nabla^{k}f_{t}^{(\alpha)}(t^{1/\alpha}x)$,
and the estimate (\ref{bound:gradientg_t-1}) follows from \eqref{eq: grad rho}.
\end{proof}
Following \cite{chen2015heat}, for a function $\varphi$ on $\Rd$,
we write 
\[
\delta_{\varphi}(x;z):=\varphi(x+z)+\varphi(x-z)-2\varphi(x).
\]

\begin{rem}
By \cite[Proposition 3.3]{Kim2017} and the same argument as in the
proof of Lemma \ref{cor:For-each-k}, we can find a constant $C=C(d,\alpha_{1},\alpha_{2},\Lambda_{1},\Lambda_{2})>0$
such that 
\begin{equation}
|\delta_{f_{t}^{(\alpha)}}(x;h)|\le C((t^{-\frac{2}{\alpha}}|h|^{2})\wedge1)(\rho_{\alpha}^{0,\alpha}(t,x\pm h)+\rho_{\alpha}^{0,\alpha}(t,x)),\quad t>0,\ x,h\in\Rd,\label{esti: delta g_t}
\end{equation}
uniformly for \textcolor{red}{${\color{black}\alpha\in[\alpha_{1},\alpha_{2}]}$}. 
\end{rem}

Since $K$ is a symmetric function, we have, for each $\varphi\in C_{b}^{2}(\Rd)$,
\begin{align*}
 & \int_{\mathbb{R}^{d}\backslash\{0\}}\left[\varphi(x+h)-\varphi(x)-\mathbf{1}_{\{|h|\le1\}}h\cdot\nabla\varphi(x)\right]\frac{K(h)}{|h|^{d+\alpha}}\mathrm{d}h\\
 & \quad=\lim_{\varepsilon\to0}\int_{\{|h|>\varepsilon\}}\left[\varphi(x+h)-\varphi(x)\right]\frac{K(h)}{|h|^{d+\alpha}}\mathrm{d}h\\
 & \quad=\frac{1}{2}\int_{\mathbb{R}^{d}\backslash\{0\}}\delta_{\varphi}(x;h)\frac{K(h)}{|h|^{d+\alpha}}\mathrm{d}h.
\end{align*}

Similarly to Lemma \ref{cor:For-each-k}, the following result follows
from \cite[Theorem 3.4]{Kim2017}.
\begin{lem}
There exists  $C=C(d,\alpha_{1},\alpha_{2},\Lambda_{1},\Lambda_{2})>0$
such that 
\begin{equation}
\int_{\Rd}\left|\delta_{f_{t}^{(\alpha)}}(x;h)\right|\cdot|h|^{-d-\alpha}\mathrm{d}h\le C\rho_{\alpha}^{0,0}(t,x),\quad x\in\Rd,\ t>0,\label{esti:f_t,integral_differ-1}
\end{equation}
\textcolor{black}{uniformly for $\alpha\in[\alpha_{1},\alpha_{2}]$.}\label{lem:frac esti f_t-1}
\end{lem}

The following lemma is crucial for the estimates that we will establish
in the next section. To prove it we will need an inequality. Let $\gamma>0$
be a constant. According to \cite[p.~277, (2.9)]{chen2015heat}, it
holds that for $t>0$ and $x,z\in\Rd$ with $|z|\le(2t^{1/\alpha})\lor\left(|x|/2\right)$, 

\begin{equation}
\left(t^{1/\alpha}+|x+z|\right)^{-\gamma}\le4^{\gamma}\left(t^{1/\alpha}+|x|\right)^{-\gamma}.\label{ineq 1: chen-1}
\end{equation}
 
\begin{lem}
There exists $C=C(d,\alpha_{1},\alpha_{2},\Lambda_{1},\Lambda_{2})>0$
such that for $\tilde{\alpha}\in[\alpha_{1},\alpha_{2}]$, $t>0$
and $x\in\Rd$,\textcolor{red}{{} }
\begin{align}
 & \int_{\Rd}\left|\delta_{f_{t}^{(\alpha)}}(x;h)\right|\cdot\left||h|^{-d-\tilde{\alpha}}-|h|^{-d-\alpha}\right|\mathrm{d}h\nonumber \\
 & \quad\le C|\alpha-\tilde{\alpha}|\left(1+|\ln t|+\mathbf{1}_{\{|x|\ge2\}}\ln|x|\right)\left[t^{\left(\alpha-\tilde{\alpha}\right)/\alpha}\vee1\right]\rho_{\alpha}^{0,0}(t,x)\nonumber \\
 & \qquad+C|\alpha-\tilde{\alpha}|\cdot\mathbf{1}_{\{|x|\ge2\}}\ln\left(|x|\right)\rho_{\tilde{\alpha}}^{0,0}(t,x).\label{esti:f_t,integral_differ}
\end{align}
Moreover, the estimate in \emph{(\ref{esti:f_t,integral_differ})}
is uniform for $\alpha,\tilde{\alpha}\in[\alpha_{1},\alpha_{2}]$.
\label{lem:frac esti f_t}
\end{lem}

\begin{proof}
Our proof is adapted from that of \cite[Theorem 2.4]{chen2015heat}
and we will also use some ideas from \cite[Proposition 4.7]{tang2007uniqueness}.
By (\ref{esti: delta g_t}), we get 
\begin{align}
 & \int_{\mathbb{R}^{d}}|\delta_{f_{t}^{(\alpha)}}(x;h)|\cdot\left||h|^{-d-\tilde{\alpha}}-|h|^{-d-\alpha}\right|\mathrm{d}h\nonumber \\
 & \quad\le c_{1}\Big(t^{-\frac{2}{\alpha}}\int_{|h|\leq t^{1/\alpha}}\rho_{\alpha}^{0,\alpha}(t,x\pm h)|h|^{2}\cdot\left||h|^{-d-\tilde{\alpha}}-|h|^{-d-\alpha}\right|\mathrm{d}h\nonumber \\
 & \qquad+\int_{|h|>t^{1/\alpha}}\rho_{\alpha}^{0,\alpha}(t,x\pm h)\left||h|^{-d-\tilde{\alpha}}-|h|^{-d-\alpha}\right|\mathrm{d}h\nonumber \\
 & \qquad+\rho_{\alpha}^{0,\alpha}(t,x)\int_{\mathbb{R}^{d}}((t^{-\frac{2}{\alpha}}|h|^{2})\wedge1)\left||h|^{-d-\tilde{\alpha}}-|h|^{-d-\alpha}\right|\mathrm{d}h\Big)\nonumber \\
 & \quad=:c_{1}(I_{1}+I_{2}+I_{3}).\label{eq: decom into I1,I2,I3}
\end{align}

Here we only consider $\tilde{\alpha}\ge\alpha$, since the case for
$\tilde{\alpha}<\alpha$ is similar and simpler. For $|h|\neq0$,
by mean value theorem, there exists $\theta\in[0,1]$ such that 
\[
|1-|h|^{\tilde{\alpha}-\alpha}|\le|h|^{\theta(\tilde{\alpha}-\alpha)}|\alpha-\tilde{\alpha}|\cdot|\ln|h||.
\]
It follows that 
\begin{equation}
|1-|h|^{\tilde{\alpha}-\alpha}|\le R^{\tilde{\alpha}-\alpha}|\alpha-\tilde{\alpha}|\cdot|\ln|h||,\quad0<|h|\le R,\label{eq: ln h, h<R}
\end{equation}
where $R\ge1$. In particular,
\begin{equation}
|1-|h|^{\tilde{\alpha}-\alpha}|\le|\alpha-\tilde{\alpha}|\cdot|\ln|h||,\quad0<|h|\le1.\label{eq: ln h, h<1}
\end{equation}
Similarly, if $|h|>\delta$ with $\delta\in(0,1]$, then 
\begin{align}
|1-|h|^{\alpha-\tilde{\alpha}}| & \le\delta^{\theta(\alpha-\tilde{\alpha})}\left|\frac{h}{\delta}\right|^{\theta(\alpha-\tilde{\alpha})}|\alpha-\tilde{\alpha}|\cdot|\ln|h||\nonumber \\
 & \le\delta^{(\alpha-\tilde{\alpha})}|\alpha-\tilde{\alpha}|\cdot|\ln|h||.\label{eq: ln h, h>delta}
\end{align}
As a special case, we have
\begin{equation}
|1-|h|^{\alpha-\tilde{\alpha}}|\le|\alpha-\tilde{\alpha}|\cdot|\ln|h||,\quad|h|>1.\label{eq: lnh, h>1}
\end{equation}
In the following we treat the cases $t\le1$ and $t>1$ separately. 

(i) Assume $0<t\le1$. Then
\begin{align}
 & \int_{|h|>t^{1/\alpha}}\left||h|^{-d-\tilde{\alpha}}-|h|^{-d-\alpha}\right|\mathrm{d}h\nonumber \\
 & \overset{(\ref{eq: ln h, h<1}),(\ref{eq: lnh, h>1})}{\le}|\alpha-\tilde{\alpha}|\left(\int_{t^{1/\alpha}<|h|\le1}|h|^{-d-\tilde{\alpha}}|\ln|h||\mathrm{d}h+\int_{|h|>1}|h|^{-d-\alpha}|\ln|h||\mathrm{d}h\right)\nonumber \\
 & \qquad\le c_{2}|\alpha-\tilde{\alpha}|\left(1+|\ln t|\right)t^{-\tilde{\alpha}/\alpha}+c_{2}|\alpha-\tilde{\alpha}|\nonumber \\
 & \qquad\le2c_{2}|\alpha-\tilde{\alpha}|\left(1+|\ln t|\right)t^{-\tilde{\alpha}/\alpha}.\label{eq1: lemma 2.9}
\end{align}
For $I_{1}$,  we have 
\begin{align*}
I_{1} & \overset{(\ref{eq: ln h, h<1})}{\le}t^{1-\frac{2}{\alpha}}\int_{|h|\leq t^{1/\alpha}}(t^{1/\alpha}+|x\pm h|)^{-d-\alpha}|h|^{2-d-\tilde{\alpha}}|\alpha-\tilde{\alpha}|\cdot|\ln|h||\mathrm{d}h\\
 & \overset{(\ref{ineq 1: chen-1})}{\le}c_{3}|\alpha-\tilde{\alpha}|t^{1-\frac{2}{\alpha}}(t^{1/\alpha}+|x|)^{-d-\alpha}\int_{|h|\leq t^{1/\alpha}}|h|^{2-d-\tilde{\alpha}}|\ln|h||\mathrm{d}h\\
 & \quad\le c_{4}|\alpha-\tilde{\alpha}|\left(1+|\ln t|\right)\rho_{\alpha}^{0,\alpha-\tilde{\alpha}}(t,x).
\end{align*}
For $I_{2}$, we need to consider 2 cases: 

(a) If $|x|\leq2t^{1/\alpha}$, then
\begin{align*}
I_{2} & \le t^{-d/\alpha}\int_{|h|>t^{1/\alpha}}\left||h|^{-d-\tilde{\alpha}}-|h|^{-d-\alpha}\right|\mathrm{d}h\\
 & \overset{(\ref{eq1: lemma 2.9})}{\le}c_{5}|\alpha-\tilde{\alpha}|\left(1+|\ln t|\right)\rho_{\alpha}^{0,\alpha-\tilde{\alpha}}(t,x).
\end{align*}

(b) If $2t^{1/\alpha}<|x|$, we break up $I_{2}$ into three parts:

\begin{align}
I_{2} & {\displaystyle =\Big(\int_{t^{1/\alpha}<|h|\le\frac{|x|}{2}}+\int_{\frac{|x|}{2}<|h|<\frac{3|x|}{2}}}\nonumber \\
 & \qquad+\int_{|h|>\frac{3|x|}{2}}\Big)\rho_{\alpha}^{0,\alpha}(t,x\pm h)\left||h|^{-d-\tilde{\alpha}}-|h|^{-d-\alpha}\right|\mathrm{d}h\nonumber \\
 & =:I_{21}+I_{22}+I_{23.}\label{eq: break I_2}
\end{align}
We have
\begin{align*}
I_{21} & \overset{(\ref{ineq 1: chen-1})}{\le}c_{6}t(t^{1/\alpha}+|x|)^{-d-\alpha}\int_{|h|>t^{1/\alpha}}\left||h|^{-d-\tilde{\alpha}}-|h|^{-d-\alpha}\right|\mathrm{d}h\\
 & \overset{(\ref{eq1: lemma 2.9})}{\le}c_{7}|\alpha-\tilde{\alpha}|\left(1+|\ln t|\right)\rho_{\alpha}^{0,\alpha-\tilde{\alpha}}(t,x).
\end{align*}
For $I_{22}$, if $|x|<2$, then
\begin{align*}
I_{22} & \overset{(\ref{eq: ln h, h<R})}{\le}c_{8}|x|^{-d-\tilde{\alpha}}|\alpha-\tilde{\alpha}|\left(1+|\ln|x||\right)\int_{\frac{|x|}{2}<|h|<\frac{3|x|}{2}}\rho_{\alpha}^{0,\alpha}(t,x\pm h)\mathrm{d}h\\
 & \overset{(\ref{esti1:rho})}{\le}c_{9}|x|^{-d-\tilde{\alpha}}|\alpha-\tilde{\alpha}|\left(1+|\ln|x||\right)\\
 & \ \le c_{10}|x|^{-d-\alpha}|x|^{\alpha-\tilde{\alpha}}|\alpha-\tilde{\alpha}|\left(1+|\ln t|\right)\\
 & \ \le c_{11}|x|^{-d-\alpha}t^{\left(\alpha-\tilde{\alpha}\right)/\alpha}|\alpha-\tilde{\alpha}|\left(1+|\ln t|\right)\\
 & \ \le c_{12}|\alpha-\tilde{\alpha}|\left(1+|\ln t|\right)\rho_{\alpha}^{0,\alpha-\tilde{\alpha}}(t,x);
\end{align*}
if $|x|\ge2$, then 
\begin{align*}
I_{22} & \overset{(\ref{eq: lnh, h>1})}{\le}c_{13}|x|^{-d-\alpha}|\alpha-\tilde{\alpha}|\left(1+\ln|x|\right)\int_{\frac{|x|}{2}\le|h|\le\frac{3|x|}{2}}\rho_{\alpha}^{0,\alpha}(t,x\pm h)\mathrm{d}h\\
 & \overset{(\ref{esti1:rho})}{\le}c_{14}|x|^{-d-\alpha}|\alpha-\tilde{\alpha}|\left(1+\ln|x|\right)\\
 & \ \le c_{15}|\alpha-\tilde{\alpha}|\left(1+\ln|x|\right)\rho_{\alpha}^{0,\alpha-\tilde{\alpha}}(t,x)\\
 & \ \le c_{15}|\alpha-\tilde{\alpha}|\left(1+\mathbf{1}_{\{|x|\ge2\}}\ln|x|\right)\rho_{\alpha}^{0,\alpha-\tilde{\alpha}}(t,x).
\end{align*}

Note that when $|h|>3|x|/2$, we have $|x\pm h|>{\displaystyle |x|/2>t^{1/\alpha}}$.
So

\begin{align*}
I_{23} & \le\int_{|h|>\frac{3|x|}{2}}t|x\pm h|{}^{-d-\alpha}\left||h|^{-d-\tilde{\alpha}}-|h|^{-d-\alpha}\right|\mathrm{d}h\\
 & \le c_{16}t|x|{}^{-d-\alpha}\int_{|h|>\frac{3|x|}{2}}\left||h|^{-d-\tilde{\alpha}}-|h|^{-d-\alpha}\right|\mathrm{d}h\\
 & \le c_{16}t|x|^{-d-\alpha}\int_{|h|>t^{1/\alpha}}\left||h|^{-d-\tilde{\alpha}}-|h|^{-d-\alpha}\right|\mathrm{d}h\\
 & \overset{(\ref{eq1: lemma 2.9})}{\le}c_{17}|\alpha-\tilde{\alpha}|\left(1+|\ln t|\right)\rho_{\alpha}^{0,\alpha-\tilde{\alpha}}(t,x).
\end{align*}

We now turn to the integral $I_{3}$. We have 
\begin{align*}
I_{3} & =\rho_{\alpha}^{0,\alpha}(t,x)\int_{|h|\le t^{1/\alpha}}t^{-\frac{2}{\alpha}}|h|^{2}\cdot\left||h|^{-d-\tilde{\alpha}}-|h|^{-d-\alpha}\right|\mathrm{d}h\\
 & \qquad+\rho_{\alpha}^{0,\alpha}(t,x)\int_{|h|>t^{1/\alpha}}\left||h|^{-d-\tilde{\alpha}}-|h|^{-d-\alpha}\right|\mathrm{d}h\\
 & \overset{(\ref{eq: ln h, h<1}),(\ref{eq1: lemma 2.9})}{\le}|\alpha-\tilde{\alpha}|t^{-\frac{2}{\alpha}}\rho_{\alpha}^{0,\alpha}(t,x)\int_{|h|\le t^{1/\alpha}}|h|^{2-d-\tilde{\alpha}}|\ln|h||\mathrm{d}h\\
 & \qquad+c_{18}|\alpha-\tilde{\alpha}|\left(1+|\ln t|\right)t^{-\tilde{\alpha}/\alpha}\rho_{\alpha}^{0,\alpha}(t,x)\\
 & \quad\le c_{19}|\alpha-\tilde{\alpha}|\left(1+|\ln t|\right)\rho_{\alpha}^{0,\alpha-\tilde{\alpha}}(t,x).
\end{align*}

By (\ref{eq: decom into I1,I2,I3}) and the above estimates we obtained
for $I_{1}$, $I_{2}$ and $I_{3}$, we see that (\ref{esti:f_t,integral_differ})
is true if $0<t\le1$.

(ii) Assume $t>1$. \textcolor{black}{In this case, note that 
\begin{equation}
t^{1-2/\alpha}\le1.\label{eq: note t>1}
\end{equation}
}For $I_{1}$,  we can apply \eqref{eq: ln h, h<R} with $R=t^{1/\alpha}$
to get 
\begin{align*}
I_{1} & \le t^{1-\frac{2}{\alpha}}\int_{|h|\leq t^{1/\alpha}}(t^{1/\alpha}+|x\pm h|)^{-d-\alpha}t^{(\tilde{\alpha}-\alpha)/\alpha}|h|^{2-d-\tilde{\alpha}}|\alpha-\tilde{\alpha}|\cdot|\ln|h||\mathrm{d}h\\
 & \overset{(\ref{ineq 1: chen-1})}{\le}c_{20}|\alpha-\tilde{\alpha}|t^{(\tilde{\alpha}-2)/\alpha}(t^{1/\alpha}+|x|)^{-d-\alpha}\int_{|h|\leq t^{1/\alpha}}|h|^{2-d-\tilde{\alpha}}|\ln|h||\mathrm{d}h\\
 & \ \le c_{21}|\alpha-\tilde{\alpha}|(1+\ln t)\rho_{\alpha}^{0,0}(t,x).
\end{align*}

For $I_{2}$, if $|x|\leq2t^{1/\alpha}$, then
\begin{align*}
I_{2} & \overset{(\ref{eq: lnh, h>1})}{\le}|\alpha-\tilde{\alpha}|t^{-d/\alpha}\int_{|h|>t^{1/\alpha}}|h|^{-d-\alpha}|\ln|h||\mathrm{d}h\\
 & \ \le c_{22}|\alpha-\tilde{\alpha}|t^{-d/\alpha}t^{-1}(1+\ln t)\le c_{23}|\alpha-\tilde{\alpha}|(1+\ln t)\rho_{\alpha}^{0,0}(t,x);
\end{align*}
if $|x|>2t^{1/\alpha}$, by breaking $I_{2}$ into $I_{21}$, $I_{22}$
and $I_{23}$ as in (\ref{eq: break I_2}), we can argue similarly
as in (i) to get 
\begin{align*}
I_{21}+I_{23} & \le c_{24}t|x|^{-d-\alpha}\int_{|h|>t^{1/\alpha}}\left||h|^{-d-\tilde{\alpha}}-|h|^{-d-\alpha}\right|\mathrm{d}h\\
 & \overset{(\ref{eq: lnh, h>1})}{\le}c_{25}|\alpha-\tilde{\alpha}|\rho_{\alpha}^{0,\alpha}(t,x)\int_{|h|>t^{1/\alpha}}|h|^{-d-\alpha}|\ln|h||\mathrm{d}h\\
 & \ \le c_{26}|\alpha-\tilde{\alpha}|\left(1+\ln t\right)\rho_{\alpha}^{0,0}(t,x)
\end{align*}
and  
\begin{align*}
I_{22} & \overset{(\ref{eq: lnh, h>1})}{\le}c_{27}|x|^{-d-\alpha}|\alpha-\tilde{\alpha}|\left(1+\ln|x|\right)\int_{\frac{|x|}{2}<|h|<\frac{3|x|}{2}}\rho_{\alpha}^{0,\alpha}(t,x\pm h)\mathrm{d}h\\
 & \overset{(\ref{esti1:rho})}{\le}c_{28}|\alpha-\tilde{\alpha}|\left(1+\ln|x|\right)\rho_{\alpha}^{0,0}(t,x)\\
 & \ \le c_{28}|\alpha-\tilde{\alpha}|\left(1+\mathbf{1}_{\{|x|\ge2\}}\ln|x|\right)\rho_{\alpha}^{0,0}(t,x).
\end{align*}

For $I_{3}$, we have 
\begin{align*}
I_{3} & \overset{(\ref{eq: ln h, h<1}),(\ref{eq: lnh, h>1})}{\le}|\alpha-\tilde{\alpha}|t^{-\frac{2}{\alpha}}\rho_{\alpha}^{0,\alpha}(t,x)\int_{|h|\le1}|h|^{2}\cdot|h|^{-d-\tilde{\alpha}}|\ln|h||\mathrm{d}h\\
 & \qquad\quad+|\alpha-\tilde{\alpha}|t^{-\frac{2}{\alpha}}\rho_{\alpha}^{0,\alpha}(t,x)\int_{1<|h|\le t^{1/\alpha}}|h|^{2}\cdot|h|^{-d-\alpha}|\ln|h||\mathrm{d}h\\
 & \qquad\quad+|\alpha-\tilde{\alpha}|\rho_{\alpha}^{0,\alpha}(t,x)\int_{|h|>t^{1/\alpha}}|h|^{-d-\alpha}|\ln|h||\mathrm{d}h\\
 & \qquad\overset{(\ref{eq: note t>1})}{\le}c_{29}|\alpha-\tilde{\alpha}|(1+\ln t)\rho_{\alpha}^{0,0}(t,x).
\end{align*}

Summing up the above estimates, we get 
\[
I_{1}+I_{2}+I_{3}\le c_{30}|\alpha-\tilde{\alpha}|\left(1+|\ln t|+\mathbf{1}_{\{|x|\ge2\}}\ln|x|\right)\rho_{\alpha}^{0,0}(t,x),
\]
which implies (\ref{esti:f_t,integral_differ}) given $t>1$. 

Finally, we would like to add one remark. As seen in the above proof,
if $\tilde{\alpha}\ge\alpha$, the second term on the right-hand of
(\ref{esti:f_t,integral_differ}) is actually not needed. However,
for the case $\tilde{\alpha}<\alpha,$ this term becomes indispensable
when we estimate $I_{2}$ for $|x|\ge2$. 

The lemma is proved.

\end{proof}

\section{Some estimates}

\textit{\emph{Let $\mathcal{A}$ be defined as in }}(\ref{eq: defi A})\textit{.}
For the remainder of this paper, we always assume that the functions
$n:\Rd\times\Rd\to(0,\infty)$ and $\alpha:\Rd\to(0,2)$ satisfy Assumption\textit{\emph{
}}\ref{assu: main}. 

Instead of $\mathcal{A}$, we first consider the operator $\mathcal{A}^{y}$
obtained by ``freezing'' the coefficient of $\mathcal{A}$ at $y\in\Rd$,
i.e., 
\[
\mathcal{A}^{y}f(x):=\int_{\mathbb{R}^{d}\backslash\{0\}}\left[f(x+h)-f(x)-\mathbf{1}_{\{|h|\le1\}}h\cdot\nabla f(x)\right]\frac{n(y,h)}{|h|^{d+\alpha(y)}}\mathrm{d}h.
\]
Then $\mathcal{A}^{y}$ is clearly the generator of a Lévy process
$\left(Z_{t}^{y}\right)_{t\ge0}$ with the characteristic exponent
\[
\psi^{y}(u)=-\int_{\mathbb{R}^{d}\setminus\{0\}}\Big(e^{iu\cdot h}-1-\mathbf{1}_{\{|h|\le1\}}iu\cdot h\Big)\frac{n(y,h)}{|h|^{d+\alpha(y)}}\mathrm{d}h.
\]
Let $f_{t}^{y}(\cdot)$ be the density function of $Z_{t}^{y}$, i.e.,
\[
f_{t}^{y}(x):=\frac{1}{(2\pi)^{d}}\int_{\mathbb{R}^{d}}e^{-iu\cdot x}e^{-t\psi^{y}(u)}\mathrm{d}u,\quad x\in\mathbb{R}^{d},\ t>0.
\]

The following lemma extends an identity in \cite[p.~9]{ChenZhang2017},
where the constant order stable-like process was considered.
\begin{lem}
It holds that for all $x,y,w\in\Rd$ and $t>0$,
\begin{align}
f_{t}^{y}(w)-f_{t}^{x}(w) & =\int_{0}^{t/2}\int_{\mathbb{R}^{d}}f_{s}^{y}(z)(\mathcal{A}^{y}-\mathcal{A}^{x})\left(f_{t-s}^{x}(w-\cdot)\right)(z)\mathrm{d}z\mathrm{d}s\nonumber \\
 & \qquad+\int_{t/2}^{t}\int_{\mathbb{R}^{d}}f_{t-s}^{x}(z)(\mathcal{A}^{y}-\mathcal{A}^{x})\left(f_{s}^{y}(w-\cdot)\right)(z)\mathrm{d}z\mathrm{d}s.\label{Prop 1, eq 1-1-1-1}
\end{align}
\end{lem}

\begin{proof}
By Fubini, we have 
\begin{align}
 & \int_{t/2}^{t}\int_{\mathbb{R}^{d}}f_{t-s}^{x}(z)(\mathcal{A}^{y}-\mathcal{A}^{x})\left(f_{s}^{y}(w-\cdot)\right)(z)\mathrm{d}z\mathrm{d}s\label{Prop1, eq 0.5}\\
 & \quad=-\frac{1}{(2\pi)^{d}}\int_{t/2}^{t}\int_{\mathbb{R}^{d}}f_{t-s}^{x}(z)\left(\int_{\Rd}\left(\psi^{y}(u)-\psi^{x}(u)\right)e^{-s\psi^{y}(u)}e^{-iu\cdot(w-z)}\mathrm{d}u\right)\mathrm{d}z\mathrm{d}s\nonumber \\
 & \quad=-\frac{1}{(2\pi)^{d}}\int_{t/2}^{t}\int_{\mathbb{R}^{d}}\left(\psi^{y}(u)-\psi^{x}(u)\right)e^{-s\psi^{y}(u)-iu\cdot w}e^{-(t-s)\psi^{x}(u)}\mathrm{d}u\mathrm{d}s\nonumber \\
 & \quad=\frac{1}{(2\pi)^{d}}\int_{\mathbb{R}^{d}}e^{-iu\cdot w-t\psi^{x}(u)}\left(e^{-t\psi^{y}(u)}e^{t\psi^{x}(u)}-e^{-t\psi^{y}(u)/2}e^{t\psi^{x}(u)/2}\right)\mathrm{d}u\nonumber \\
 & \quad=f_{t}^{y}(w)-\frac{1}{(2\pi)^{d}}\int_{\mathbb{R}^{d}}e^{-iu\cdot w-t\psi^{x}(u)/2}e^{-t\psi^{y}(u)/2}\mathrm{d}u.\label{Prop 1, eq 1}
\end{align}
By the change of variables $s':=t-s$ and interchanging the roles
of $y$ and $x$ in (\ref{Prop1, eq 0.5}), we obtain
\begin{align*}
 & \int_{0}^{t/2}\int_{\mathbb{R}^{d}}f_{s}^{y}(z)(\mathcal{A}^{x}-\mathcal{A}^{y})\left(f_{t-s}^{x}(w-\cdot)\right)(z)\mathrm{d}z\mathrm{d}s\\
 & \quad=f_{t}^{x}(w)-\frac{1}{(2\pi)^{d}}\int_{\mathbb{R}^{d}}e^{-iu\cdot w-t\psi^{y}(u)/2}e^{-t\psi^{x}(u)/2}\mathrm{d}u,
\end{align*}
which, together with (\ref{Prop 1, eq 1}), implies (\ref{Prop 1, eq 1-1-1-1}). 
\end{proof}
\begin{lem}
\label{lem: esti f^y-f^x} There exists $C=C(d,\underline{\alpha},\overline{\alpha},\kappa_{1},\kappa_{2})>0$
such that for all $t\in(0,1/2]$, ${\color{red}{\normalcolor x,y\in\Rd}}$
and $w\in\Rd$ with $0<|w|\le1$,
\begin{align*}
|f_{t}^{y}(w)-f_{t}^{x}(w)| & \le C\left(t^{1-|\alpha(x)-\alpha(y)|/\underline{\alpha}}|\ln t|\beta(|x-y|)+t\psi(|x-y|)\right)\\
 & \quad\times\exp\left(|\alpha(x)-\alpha(y)|\cdot|\ln|w||\right)\cdot\left\{ \rho_{\alpha(x)}^{0,0}(t,w)+\rho_{\alpha(y)}^{0,0}(t,w)\right\} ,
\end{align*}
where $\beta$ and $\psi$ are defined in the same way as in\emph{
}Assumption \emph{\ref{assu: main}}. 
\end{lem}

\begin{proof}
We denote the first and second term on the right-hand side of (\ref{Prop 1, eq 1-1-1-1})
by $I(t,x,y,w)$ and $J(t,x,y,w)$, respectively. It suffices to establish
the asserted estimates for $|I|$ and $|J|$. Here we only treat $I(t,x,y,w)$,
since the case for $J(t,x,y,w)$ is similar.

By the symmetry of $n(x,\cdot)$ and $n(y,\cdot)$, we see that 
\begin{equation}
I(t,x,y,w)=\int_{0}^{t/2}\int_{\mathbb{R}^{d}}f_{s}^{y}(z)\left[(\mathcal{A}^{y}-\mathcal{A}^{x})f_{t-s}^{x}\right](w-z)\mathrm{d}z\mathrm{d}s.\label{eq: new defi, I}
\end{equation}
Noting
\[
\frac{n(x,h)}{|h|^{d+\alpha(x)}}-\frac{n(y,h)}{|h|^{d+\alpha(y)}}=\frac{n(x,h)}{|h|^{d+\alpha(x)}}-\frac{n(y,h)}{|h|^{d+\alpha(x)}}+\frac{n(y,h)}{|h|^{d+\alpha(x)}}-\frac{n(y,h)}{|h|^{d+\alpha(y)}},
\]
we have
\begin{align}
\left|\left[(\mathcal{A}^{y}-\mathcal{A}^{x})f_{s}^{x}\right](w)\right| & \le\kappa_{2}\int_{\mathbb{R}^{d}\backslash\{0\}}|\delta_{f_{s}^{x}}(w;h)|\cdot\left||h|^{-d-\alpha(x)}-|h|^{-d-\alpha(y)}\right|\mathrm{d}h\nonumber \\
 & \qquad+\int_{\mathbb{R}^{d}\backslash\{0\}}|\delta_{f_{s}^{x}}(w;h)|\cdot\frac{|n(x,h)-n(y,h)|}{|h|^{d+\alpha(x)}}\mathrm{d}h\nonumber \\
 & =:\kappa_{2}F_{1}(s,x,y,w)+F_{2}(s,x,y,w),\label{eq0.5: Lemma 3.2}
\end{align}
For $0<t/2\le s\le t\le1/2$, it follows from Lemma \ref{lem:frac esti f_t}
and the definition of $\beta$ that 
\begin{align}
F_{1}(s,x,y,w) & \le c_{1}|\alpha(x)-\alpha(y)|\left(|\ln s|+\mathbf{1}_{\{|w|\ge2\}}\ln|w|\right)s^{-|\alpha(x)-\alpha(y)|/\underline{\alpha}}\nonumber \\
 & \quad\times\left\{ \rho_{\alpha(x)}^{0,0}(s,w)+\rho_{\alpha(y)}^{0,0}(s,w)\right\} \nonumber \\
 & \le c_{2}\beta(|x-y|)\left(|\ln t|+\mathbf{1}_{\{|w|\ge2\}}\ln|w|\right)t^{-|\alpha(x)-\alpha(y)|/\underline{\alpha}}\nonumber \\
 & \quad\times t^{-1}\left\{ \rho_{\alpha(x)}^{0,\alpha(x)}(s,w)+\rho_{\alpha(y)}^{0,\alpha(y)}(s,w)\right\} .\label{eq1: Lemma 3.2}
\end{align}
Since $f_{s}^{y}(z)\le c_{3}\rho_{\alpha(y)}^{0,\alpha(y)}(s,z)$
by (\ref{bound:gradientg_t-1}), it follows from (\ref{eq: Lemma 2.2}),
(\ref{eq: Lemma 2.4}) and (\ref{eq1: Lemma 3.2}) that for $0<s\le t/2\le1/4$
and $|w|\le1$, 
\begin{align}
 & \int_{\mathbb{R}^{d}}f_{s}^{y}(z)F_{1}(t-s,x,y,w-z)\mathrm{d}z\nonumber \\
 & \quad\le c_{4}t^{-1-|\alpha(x)-\alpha(y)|/\underline{\alpha}}\beta(|x-y|)\int_{\mathbb{R}^{d}}\rho_{\alpha(y)}^{0,\alpha(y)}(s,z)\left(|\ln t|+\mathbf{1}_{\{|w-z|\ge2\}}\ln|w-z|\right)\nonumber \\
 & \qquad\times\left\{ \rho_{\alpha(x)}^{0,\alpha(x)}(t-s,w-z)+\rho_{\alpha(y)}^{0,\alpha(y)}(t-s,w-z)\right\} \mathrm{d}z\nonumber \\
 & \quad\le c_{5}t^{-|\alpha(x)-\alpha(y)|/\underline{\alpha}}\left(1+|\ln s|+|\ln t|\right)\beta(|x-y|)\exp\{|\alpha(x)-\alpha(y)|\cdot|\ln|w||\}\nonumber \\
 & \qquad\times\left\{ \rho_{\alpha(x)}^{0,0}(t,w)+\rho_{\alpha(y)}^{0,0}(t,w)\right\} .\label{eq2: Lemma 3.2}
\end{align}
Similarly, for $0<s\le t/2\le1/4$, we obtain
\begin{align}
 & \int_{\mathbb{R}^{d}}f_{s}^{y}(z)F_{2}(t-s,x,y,w-z)\mathrm{d}z\nonumber \\
 & \quad\le c_{6}\psi(|x-y|)\exp\{|\alpha(x)-\alpha(y)|\cdot|\ln|w||\}\cdot\left\{ \rho_{\alpha(x)}^{0,0}(t,w)+\rho_{\alpha(y)}^{0,0}(t,w)\right\} .\label{eq3: Lemma 3.2}
\end{align}
Since (\ref{eq: new defi, I}) and (\ref{eq0.5: Lemma 3.2}) hold,
the desired estimate for $|I(t,x,y,w)|$ follows when we integrate
(\ref{eq2: Lemma 3.2}) and (\ref{eq3: Lemma 3.2}) with respect to
$s$ from $0$ to $t/2$. The lemma is proved. 
\end{proof}
Based on the last lemma, we are now ready to prove the following. 
\begin{prop}
\label{lem: integral f_t^y - f_t^x}For each $x\in\Rd$, we have 
\[
\lim_{t\to0}\int_{|y-x|\le1}|f_{t}^{y}(y-x)-f_{t}^{x}(y-x)|\mathrm{d}y=0.
\]
\end{prop}

\begin{proof}
Let $t\in(0,1/2]$. Define $D_{1}:=\left\{ y:|y-x|\le t^{1/2}\right\} $
and 
\[
D_{2}:=\left\{ y:t^{1/2}<|y-x|\le1\right\} .
\]
It follows from Assumption \ref{assu: main}(d) that for all $x,y\in\Rd$
with $|y-x|\le1$,
\begin{equation}
\exp\{|\alpha(y)-\alpha(x)|\cdot|\ln|y-x||\}\le\exp\{\beta(|y-x|)|\ln|y-x||\}\le c_{1}<\infty.\label{eq0 : prop 3.3}
\end{equation}

``Step 1'': On $D_{1}$, we have 
\begin{align}
t^{-|\alpha(x)-\alpha(y)|/\underline{\alpha}} & =\exp\left\{ \underline{\alpha}^{-1}|\alpha(x)-\alpha(y)|\ln(t^{-1})\right\} \nonumber \\
 & \le\exp\left\{ 2\underline{\alpha}^{-1}\beta(|x-y|)|\ln|x-y||\right\} \overset{(\ref{eq0 : prop 3.3})}{\le}c_{2}<\infty.\label{eq0.5 : prop 3.3}
\end{align}
By (\ref{eq0 : prop 3.3}), (\ref{eq0.5 : prop 3.3}) and Lemma \ref{lem: esti f^y-f^x},
we see that for $y\in D_{1}$,
\begin{align}
 & |f_{t}^{y}(y-x)-f_{t}^{x}(y-x)|\nonumber \\
 & \le c_{3}t|\ln t|\beta(|x-y|)\rho_{\alpha(x)}^{0,0}(t,y-x)+c_{3}t|\ln t|\beta(|x-y|)\rho_{\alpha(y)}^{0,0}(t,y-x)\nonumber \\
 & \quad+c_{3}t\psi(|x-y|)\rho_{\alpha(x)}^{0,0}(t,y-x)+c_{3}t\psi(|x-y|)\rho_{\alpha(y)}^{0,0}(t,y-x)\nonumber \\
 & =:c_{3}I_{1}(t)+c_{3}I_{2}(t)+c_{3}I_{3}(t)+c_{3}I_{4}(t).\label{eq0.7 : prop 3.3}
\end{align}
If $|y-x|\le t^{1/\alpha(y)}$, then 
\begin{align}
I_{2}(t)+I_{4}(t) & \le t\left[\left(\ln t^{-1}\right)\beta(|x-y|)+\psi(|x-y|)\right](t^{1/\alpha(y)})^{-d-\alpha(y)}\nonumber \\
 & \le c_{4}\left[\left(\ln|x-y|\right)\beta(|x-y|)+\psi(|x-y|)\right]\left(|x-y|\right)^{-d}.\label{eq1  : prop 3.3}
\end{align}
Note that $t\mapsto t\ln t^{-1}$ is increasing on $(0,1/e)$. If
$t$ is sufficiently small and $t^{1/\alpha(y)}<|y-x|\le t^{1/2}$,
then 
\begin{align}
I_{2}(t)+I_{4}(t) & \le\left[t\left(\ln t^{-1}\right)\beta(|x-y|)+t\psi(|x-y|)\right](|x-y|)^{-d-\alpha(y)}\nonumber \\
 & \le c_{5}\left[\left(\ln|x-y|\right)\beta(|x-y|)+\psi(|x-y|)\right]\left(|x-y|\right)^{-d}.\label{eq2 : prop 3.3}
\end{align}
It follows from (\ref{eq1  : prop 3.3}) and (\ref{eq2 : prop 3.3})
that 
\[
\int_{D_{1}}\left[I_{2}(t)+I_{4}(t)\right]\mathrm{d}y\le c_{6}\int_{0}^{t^{1/2}}\frac{\psi(r)+\beta(r)|\ln r|}{r}\mathrm{d}r\to0,\quad\mbox{as \ensuremath{t\to0},}
\]
where the convergence of the integral to $0$ follows by Assumption
\ref{assu: main}(b) and (d). The cases for $I_{1}$ and $I_{3}$
are similar, so, by (\ref{eq0.7 : prop 3.3}), 
\begin{equation}
\lim_{t\to0}\int_{D_{1}}|f_{t}^{y}(y-x)-f_{t}^{x}(y-x)|\mathrm{d}y=0.\label{eq 3 : prop 3.3}
\end{equation}

``Step 2'': On $D_{2}$, we have 
\begin{align}
 & \int_{D_{2}}|f_{t}^{y}(y-x)-f_{t}^{x}(y-x)|\mathrm{d}y\nonumber \\
 & \quad\le\int_{t^{1/2}<|y-x|\le1}\left[f_{t}^{y}(y-x)+f_{t}^{x}(y-x)\right]\mathrm{d}y\nonumber \\
 & \quad\le c_{7}\int_{t^{1/2}<|y-x|\le1}\frac{t}{|x-y|^{d+\overline{\alpha}}}\mathrm{d}y\le c_{8}t^{1-\overline{\alpha}/2}\to0,\quad\mbox{as \ensuremath{t\to0.}}\label{eq 4 : prop 3.3}
\end{align}

The assertion now follows by (\ref{eq 3 : prop 3.3}) and (\ref{eq 4 : prop 3.3}).
\end{proof}
In the rest of this section we establish some estimates that we will
use in the proof of Theorem \ref{thm: main}.

Define

\begin{equation}
q(t,x,y):=f_{t}^{y}(y-x),\quad t>0,\ x,y\in\Rd.\label{eq: defi q}
\end{equation}
The function $q(t,x,y)$ is usually called the parametrix. Let
\begin{align}
F(t,x,y):= & \left(\mathcal{A}-\mathcal{A}^{y}\right)q(t,\cdot,y)(x)\nonumber \\
= & \int_{\mathbb{R}^{d}\backslash\{0\}}\Big[q(t,x+h,y)-q(t,x,y)\nonumber \\
 & \qquad-\mathbf{1}_{\{|h|\le1\}}h\cdot\nabla_{x}q(t,x,y)\Big]\left(\frac{n(x,h)}{|h|^{d+\alpha(x)}}-\frac{n(y,h)}{|h|^{d+\alpha(y)}}\right)\mathrm{d}h.\label{eq: defi F}
\end{align}

Similarly to (\ref{eq0.5: Lemma 3.2}), we have 
\begin{align}
|F(t,x,y)| & \le\kappa_{2}\int_{\mathbb{R}^{d}\backslash\{0\}}|\delta_{f_{t}^{y}(y-\cdot)}(x;h)|\cdot\left||h|^{-d-\alpha(x)}-|h|^{-d-\alpha(y)}\right|\mathrm{d}h\nonumber \\
 & \qquad+\int_{\mathbb{R}^{d}\backslash\{0\}}|\delta_{f_{t}^{y}(y-\cdot)}(x;h)|\cdot\frac{|n(x,h)-n(y,h)|}{|h|^{d+\alpha(y)}}\mathrm{d}h\nonumber \\
 & =:\kappa_{2}F_{1}(t,x,y)+F_{2}(t,x,y).\label{eq: decom F}
\end{align}
Note that $\delta_{f_{t}^{y}(y-\cdot)}(x;h)=\delta_{f_{t}^{y}}(y-x;h)$.
By Lemmas \ref{lem:frac esti f_t-1} and \ref{lem:frac esti f_t},
we get that for $t>0$ and $x,y\in\Rd$, 
\begin{align}
F_{1}(t,x,y) & \le c\beta\left(|x-y|\right)\left(1+|\ln t|+\mathbf{1}_{\{|y-x|\ge2\}}\ln|y-x|\right)\nonumber \\
 & \quad\times\big[t^{\left(\alpha(y)-\alpha(x)\right)/\alpha(y)}\vee1\big]\rho_{\alpha(y)}^{0,0}(t,y-x)\nonumber \\
 & \ +c\beta\left(|x-y|\right)\mathbf{1}_{\{|y-x|\ge2\}}\ln\left(|y-x|\right)\rho_{\alpha(x)}^{0,0}(t,y-x)\label{eq: esti, F_1}
\end{align}
and
\begin{equation}
F_{2}(t,x,y)\le c\psi\left(|x-y|\right)\rho_{\alpha(y)}^{0,0}(t,y-x),\label{eq: esti, F_2}
\end{equation}
where $c=c(d,\alpha_{1},\alpha_{2},\Lambda_{1},\Lambda_{2})>0$ is
a constant.

As we will see later, the essential ingredient to prove Theorem \ref{thm: main}
is to show that 
\[
\sup_{x\in\Rd}\int_{0}^{\infty}\int_{\Rd}e^{-\lambda t}|F(t,x,y)|\mathrm{d}y\mathrm{d}t\le\frac{1}{2}
\]
for sufficiently large $\lambda>0$. We will achieve this in a few
steps. First, we estimate the integral $\int_{\Rd}|F(t,x,y)\mathrm{|d}y$
when $t$ is away from $0$.
\begin{lem}
\label{lem: esti delta >1}Suppose $0<\delta<1$. There exists $C=C(\delta,d,\underline{\alpha},\overline{\alpha},\kappa_{1},\kappa_{2})>0$
such that for all $x\in\Rd$ and $t\ge\delta$, 
\[
\int_{\Rd}|F(t,x,y)\mathrm{|d}y\le C\left(1+|\ln t|\right)t^{\left(d+2\right)/\underline{\alpha}}.
\]
 
\end{lem}

\begin{proof}
We split $\Rd$ as the union of $\left\{ y:|y-x|<t^{1/\underline{\alpha}}\right\} $
and $\left\{ y:t^{1/\underline{\alpha}}\le|y-x|\right\} $. Note that
for $t\ge\delta$,
\begin{equation}
\rho_{\alpha(y)}^{0,0}(t,y-x)+\rho_{\alpha(x)}^{0,0}(t,y-x)\le t^{-1}\left(t^{-d/\alpha(x)}+t^{-d/\alpha(y)}\right)\le2\delta^{-1-d/\underline{\alpha}}.\label{eq 1: lemma 3.4}
\end{equation}
Since $\beta$ and $\psi$ are bounded by Assumption \ref{assu: main}(a)
and (c), it follows from (\ref{eq: esti, F_1}), (\ref{eq: esti, F_2})
and (\ref{eq 1: lemma 3.4}) that for $t\in[\delta,\infty)$,
\begin{align}
\int_{|y-x|<t^{1/\underline{\alpha}}}|F(t,x,y)|\mathrm{d}y & \le c_{1}\int_{|y-x|<t^{1/\underline{\alpha}}}t^{2/\underline{\alpha}}\left(1+|\ln t|+|\ln|y-x||\right)\mathrm{d}y\nonumber \\
 & \le c_{2}\left(1+|\ln t|\right)t^{\left(d+2\right)/\underline{\alpha}}.\label{eq 2: lemma 3.4}
\end{align}
Similarly, for $t\in[\delta,\infty)$,
\begin{align}
 & \int_{|y-x|\ge t^{1/\underline{\alpha}}}|F(t,x,y)|\mathrm{d}y\nonumber \\
 & \quad\le c_{3}\left(1+|\ln t|\right)t^{2/\underline{\alpha}}\int_{|y-x|\ge t^{1/\underline{\alpha}}}\left(|y-x|^{-d-\underline{\alpha}}+|y-x|^{-d-\overline{\alpha}}\right)\left(1+|\ln|y-x||\right)\mathrm{d}y\nonumber \\
 & \quad\le c_{4}\left(1+|\ln t|\right)t^{2/\underline{\alpha}}\int_{t^{1/\underline{\alpha}}}^{\infty}\left(r^{-1-\underline{\alpha}}+r^{-1-\overline{\alpha}}\right)(1+|\ln r|)\mathrm{d}r\nonumber \\
 & \quad\le c_{5}\left(1+|\ln t|\right)^{2}t^{2/\underline{\alpha}}(t^{-1}+t^{-\overline{\alpha}/\underline{\alpha}})\le c_{6}\left(1+|\ln t|\right)t^{\left(d+2\right)/\underline{\alpha}}.\label{eq 3: lemma 3.4}
\end{align}
Combining (\ref{eq 2: lemma 3.4}) and (\ref{eq 3: lemma 3.4}) gives
the assertion.
\end{proof}

\section{A special case: $\overline{\alpha}<2\underline{\alpha}$ }

In this section we will prove the statement of Theorem \ref{thm: main}
under the additional condition that 
\begin{equation}
\overline{\alpha}<2\underline{\alpha},\label{eq: extral cond.}
\end{equation}
where $\underline{\alpha}$ and $\overline{\alpha}$ are as in Assumption\textit{
}\ref{assu: main}(c). In the next section we will show that this
extra requirement is not necessary by some localization argument.

Recall that $F(t,x,y)$ is defined in (\ref{eq: defi F}). 
\begin{lem}
\label{lem: delta to 0} Assume that \emph{(\ref{eq: extral cond.})}
is true. Then 
\[
\lim_{\delta\to0}\left(\sup_{x\in\Rd}\int_{0}^{\delta}\int_{\Rd}|F(t,x,y)\mathrm{|d}y\mathrm{d}t\right)=0.
\]
 
\end{lem}

\begin{proof}
Let $\varepsilon>0$ be arbitrary. We claim that we can find a sufficiently
small constant $c\in(0,1/2)$ such that
\begin{equation}
\int_{0}^{c}\frac{\psi(r)+\beta(r)|\ln r|}{r}\mathrm{d}r<\varepsilon\label{eq: choice of c}
\end{equation}
and 
\begin{equation}
\delta\mapsto\delta^{1-|\alpha(x)-\alpha(y)|/\underline{\alpha}}\ln\delta^{-1}\quad\mbox{is increasing on \ensuremath{(0,c^{\underline{\alpha}}],} \ for all \ensuremath{x,y\in\Rd.} }\label{eq: condition for the constant c}
\end{equation}
Indeed, (\ref{eq: choice of c}) is easily fulfilled by Assumption
\ref{assu: main}(b) and (d). To see the existence of $c$ as in (\ref{eq: condition for the constant c}),
we only need to note 
\begin{equation}
|\alpha(x)-\alpha(y)|/\underline{\alpha}\le\left(\overline{\alpha}-\underline{\alpha}\right)/\underline{\alpha}<1,\quad x,y\in\Rd,\label{eq0.5: lemma 4.1}
\end{equation}
which implies that the derivative of the function in (\ref{eq: condition for the constant c})
is positive for small enough $\delta$, say, smaller than a constant
$\delta_{0}>0$. Moreover, by (\ref{eq0.5: lemma 4.1}), $\delta_{0}$
can be chosen to be independent of $x,y\in\Rd$. 

In the rest of the proof we consider 
\begin{equation}
\delta\in(0,c^{\overline{\alpha}}/2]\subset(0,1/2].\label{eq: range delta, lemma 4.1}
\end{equation}
Define $D_{1}:=\left\{ y:0<|y-x|^{\alpha(y)}<\delta\right\} $, $D_{2}:=\left\{ y:\delta\le|y-x|^{\alpha(y)}<c^{\alpha(y)}\right\} $
and $D_{3}:=\left\{ y:|y-x|\ge c\right\} $. Then
\begin{align*}
\int_{0}^{\delta}\int_{\Rd}|F(t,x,y)\mathrm{|d}y\mathrm{d}t & =\left(\int_{D_{1}}\int_{0}^{\delta}+\int_{D_{2}}\int_{0}^{\delta}+\int_{D_{3}}\int_{0}^{\delta}\right)|F(t,x,y)|\mathrm{d}t\mathrm{d}y\\
 & =:I_{\delta}(x)+J_{\delta}(x)+H_{\delta}(x).
\end{align*}

We now treat $I_{\delta}(x)$, $J_{\delta}(x)$ and $H_{\delta}(x)$
separately. We first make two observations. First, it follows from
(\ref{eq: esti, F_1}) and (\ref{eq: esti, F_2}) that for $|y-x|\le1$
and $0<t\le1/2$,
\begin{equation}
|F(t,x,y)|\le c_{1}\left[\psi\left(|x-y|\right)+t^{-|\alpha(x)-\alpha(y)|/\underline{\alpha}}|\ln t|\beta\left(|x-y|\right)\right]\varrho_{\alpha(y)}^{0,0}(t,x-y).\label{eq1: lemma 4.1}
\end{equation}
Second, as in (\ref{eq0 : prop 3.3}), if $|y-x|\le1$, then
\begin{equation}
|x-y|^{-\alpha(y)|\alpha(x)-\alpha(y)|/\underline{\alpha}}\le\exp\left(2\underline{\alpha}^{-1}\beta\left(|x-y|\right)|\ln|x-y||\right)\le c_{2}<\infty.\label{eq2: lemma 4.1}
\end{equation}

(i) If $y\in D_{1}$, then $|y-x|^{\alpha(y)}<\delta\overset{(\ref{eq: range delta, lemma 4.1})}{\le}1/2$
and 
\begin{equation}
|\ln|y-x||\ge c_{3}>0.\label{eq2.5: lemma 4.1}
\end{equation}
Therefore, for $y\in D_{1}$, we have 

\begin{align}
 & \int_{0}^{\delta}t^{-|\alpha(x)-\alpha(y)|/\underline{\alpha}}|\ln t|\varrho_{\alpha(y)}^{0,0}(t,x-y)\mathrm{d}t\nonumber \\
 & \quad\le\int_{0}^{|y-x|^{\alpha(y)}}t^{-|\alpha(x)-\alpha(y)|/\underline{\alpha}}|\ln t|\cdot|x-y|^{-d-\alpha(y)}\mathrm{d}t\nonumber \\
 & \qquad+\int_{|y-x|^{\alpha(y)}}^{\delta}t^{-|\alpha(x)-\alpha(y)|/\underline{\alpha}}|\ln t|(t^{1/\alpha(y)})^{-d-\alpha(y)}\mathrm{d}t\nonumber \\
 & \quad\overset{(\ref{eq0.5: lemma 4.1})}{\le}c_{4}|x-y|^{-d}\left(1+\left|\ln|x-y|\right|\right)|x-y|^{-\alpha(y)|\alpha(x)-\alpha(y)|/\underline{\alpha}}\nonumber \\
 & \quad\overset{(\ref{eq2: lemma 4.1}),(\ref{eq2.5: lemma 4.1})}{\le}c_{2}c_{4}(1+c_{3}^{-1})|x-y|^{-d}\left|\ln|x-y|\right|,\label{eq3: lemma 4.1}
\end{align}
and, similarly, 
\begin{equation}
\int_{0}^{\delta}\varrho_{\alpha(y)}^{0,0}(t,y-x)\mathrm{d}t\le c_{5}|x-y|^{-d}.\label{eq4: lemma 4.1}
\end{equation}
Note that $\delta\le1/2$. It follows from (\ref{eq1: lemma 4.1}),
(\ref{eq3: lemma 4.1}) and (\ref{eq4: lemma 4.1}) that 

\begin{align}
I_{\delta}(x) & \le c_{6}\int_{D_{1}}|x-y|^{-d}\left(\psi\left(|x-y|\right)+\beta\left(|x-y|\right)\left|\ln|x-y|\right|\right)\mathrm{d}y\nonumber \\
 & \le c_{7}\int_{0}^{\delta^{1/\overline{\alpha}}}\frac{\psi(r)+\beta(r)|\ln r|}{r}\mathrm{d}r.\label{eq: I, lemma 4.1}
\end{align}

(ii) If $y\in D_{2}$, then $\delta\le|y-x|^{\alpha(y)}<c^{\alpha(y)}\le c^{\underline{\alpha}}$
and 
\begin{align}
 & \int_{0}^{\delta}t^{-|\alpha(x)-\alpha(y)|/\underline{\alpha}}|\ln t|\mathrm{d}t\overset{(\ref{eq0.5: lemma 4.1})}{\le}c_{8}\delta^{1-|\alpha(x)-\alpha(y)|/\underline{\alpha}}\left(1+\ln\delta^{-1}\right)\nonumber \\
 & \qquad\overset{(\ref{eq: range delta, lemma 4.1})}{\le}c_{9}\delta^{1-|\alpha(x)-\alpha(y)|/\underline{\alpha}}\ln\delta^{-1}\nonumber \\
 & \qquad\overset{(\ref{eq: condition for the constant c})}{\le}c_{10}|x-y|^{\alpha(y)-\alpha(y)|\alpha(x)-\alpha(y)|/\underline{\alpha}}|\ln|x-y||\nonumber \\
 & \qquad\overset{(\ref{eq2: lemma 4.1})}{\le}c_{11}|x-y|^{\alpha(y)}|\ln|x-y||.\label{eq5: lemma 4.1}
\end{align}
Therefore, for $y\in D_{2}$, 
\begin{align*}
 & \int_{0}^{\delta}|F(t,x,y)\mathrm{|d}t\\
 & \quad\overset{(\ref{eq1: lemma 4.1})}{\le}c_{1}|x-y|^{-d-\alpha(y)}\int_{0}^{\delta}\left[\psi\left(|x-y|\right)+t^{-|\alpha(x)-\alpha(y)|/\underline{\alpha}}|\ln t|\beta\left(|x-y|\right)\right]\mathrm{d}t\\
 & \quad\overset{(\ref{eq5: lemma 4.1})}{\le}c_{12}|x-y|^{-d}\left[\delta|x-y|^{-\alpha(y)}\psi\left(|x-y|\right)+\beta\left(|x-y|\right)\left|\ln|y-x|\right|\right]\\
 & \quad\ \le\ c_{12}|x-y|^{-d}\left[\psi\left(|x-y|\right)+\beta\left(|x-y|\right)\left|\ln|y-x|\right|\right].
\end{align*}
So
\begin{align}
J_{\delta}(x) & \le c_{12}\int_{|x-y|\le c}|x-y|^{-d}\left[\psi\left(|x-y|\right)+\beta\left(|x-y|\right)\left|\ln|y-x|\right|\right]\mathrm{d}y\nonumber \\
 & \le c_{13}\int_{0}^{c}\frac{\psi(r)+\beta(r)|\ln r|}{r}\mathrm{d}r\overset{(\ref{eq: choice of c})}{\le}c_{13}\varepsilon.\label{eq: J, lemma 4.1}
\end{align}

(iii) For $y\in D_{3}$ and $0<t\le\delta\le1/2$, it follows from
(\ref{eq: esti, F_1}) and (\ref{eq: esti, F_2}) that
\[
|F(t,x,y)|\le c_{14}t^{-\left(\overline{\alpha}-\underline{\alpha}\right)/\underline{\alpha}}\left(1+|\ln t|\right)\left(1+|\ln|y-x||\right)\left[|y-x|^{-d-\underline{\alpha}}+|y-x|^{-d-\overline{\alpha}}\right].
\]
So

\begin{align}
H_{\delta}(x) & \le c_{14}\int_{|y-x|\ge c}\left(1+|\ln|y-x||\right)\left[|y-x|^{-d-\underline{\alpha}}+|y-x|^{-d-\overline{\alpha}}\right]\mathrm{d}y\nonumber \\
 & \qquad\times\int_{0}^{\delta}t^{-\left(\overline{\alpha}-\underline{\alpha}\right)/\underline{\alpha}}\left(1+|\ln t|\right)\mathrm{d}t\to0,\quad\mbox{as \ensuremath{\delta\to}0, }\label{eq: H, lemma 4.1}
\end{align}
where the convergence in (\ref{eq: H, lemma 4.1}) follows from the
assumption that $\overline{\alpha}<2\underline{\alpha}$.

We emphasize that the above constants $c_{1},\cdots,c_{13}$ depend
only on $d,\underline{\alpha},\overline{\alpha},\kappa_{1},\kappa_{2}$
and $\beta$. It follows from (\ref{eq: I, lemma 4.1}), (\ref{eq: J, lemma 4.1})
and (\ref{eq: H, lemma 4.1}) that 
\[
\limsup_{\delta\to0}\left(\sup_{x\in\Rd}\int_{0}^{\delta}\int_{\Rd}|F(t,x,y)\mathrm{|d}y\mathrm{d}t\right)\le c_{13}\varepsilon.
\]
Since $\varepsilon>0$ is arbitrary, the assertion follows. 
\end{proof}
Now, we can combine the estimates in Lemmas \ref{lem: esti delta >1}
and \ref{lem: delta to 0} to get the following. 
\begin{prop}
\label{prop: <12}Under the assumptions of Lemma \emph{\ref{lem: delta to 0}},
there exists $\lambda_{0}>0$ such that for all $\lambda\geq\lambda_{0}$,
\[
\sup_{x\in\Rd}\int_{0}^{\infty}\int_{\Rd}e^{-\lambda t}|F(t,x,y)\mathrm{|d}y\mathrm{d}t\le\frac{1}{2}.
\]
\end{prop}

\begin{proof}
According to Lemma \ref{lem: delta to 0}, there exists sufficiently
small $\delta_{0}>0$ such that 
\begin{equation}
\sup_{x\in\Rd}\int_{0}^{\delta_{0}}\int_{\Rd}e^{-\lambda t}|F(t,x,y)\mathrm{|d}y\mathrm{d}t<\frac{1}{4},\quad\mbox{for all \ensuremath{\lambda>0}}.\label{eq1: prop 3.6}
\end{equation}
By Lemma \ref{lem: esti delta >1}, there exists $c_{1}=c_{1}(\delta_{0},d,\underline{\alpha},\overline{\alpha},\kappa_{1},\kappa_{2})>0$
such that for all $x\in\Rd$ and $t\ge\delta_{0}$, 
\[
\int_{\Rd}|F(t,x,y)\mathrm{|d}y\le c_{1}\left(1+|\ln t|\right)t^{\left(d+2\right)/\underline{\alpha}}.
\]
So 
\begin{equation}
\sup_{x\in\Rd}\int_{\delta_{0}}^{\infty}\int_{\Rd}e^{-\lambda t}|F(t,x,y)\mathrm{|d}y\mathrm{d}t\le c_{1}\int_{\delta_{0}}^{\infty}e^{-\lambda t}\left(1+|\ln t|\right)t^{\left(d+2\right)/\underline{\alpha}}\mathrm{d}t,\label{eq2: prop 3.6}
\end{equation}
where the right-hand side converges to $0$ as $\lambda\to\infty$.
Now choose $\lambda_{0}>0$ so that 
\begin{equation}
c_{1}\int_{\delta_{0}}^{\infty}e^{-\lambda t}\left(1+|\ln t|\right)t^{\left(d+2\right)/\underline{\alpha}}\mathrm{d}t\le\frac{1}{4},\quad\lambda\geq\lambda_{0}.\label{eq3: prop 3.6}
\end{equation}
Combining (\ref{eq1: prop 3.6}), (\ref{eq2: prop 3.6}) and (\ref{eq3: prop 3.6})
gives the assertion. 
\end{proof}
We are now ready to prove the following special case of Theorem \ref{thm: main}. 
\begin{prop}
\label{prop:Let--be main}Let $\mathcal{A}$ be as in \emph{(\ref{eq: defi A})},
and suppose Assumption \emph{\ref{assu: main}} holds. Further, assume
that \eqref{eq: extral cond.} is true. Then for each $x\in\Rd$,
the martingale problem for the operator $\mathcal{A}$ starting at
$x$ has at most one solution.
\end{prop}

\begin{proof}
In view of Propositions \ref{lem: integral f_t^y - f_t^x} and \ref{prop: <12},
the same proof as in \cite[Section 3]{MR2642351} applies also to
our case. However, for the reader's convenience, we spell out the
details here. 

Suppose $\mathbf{P}_{1},\mathbf{P}_{2}$ are two solutions to the
martingale problem for $\mathcal{A}$ started at a point $x_{0}\in\Rd$.
For $\varphi\in C_{b}(\Rd)$, define 
\[
S_{\lambda}^{i}\varphi:=\mathrm{\mathbf{E}}_{i}\int_{0}^{\infty}e^{-\lambda t}\varphi(X_{t})\mathrm{d}t,\ i=1,2,
\]
and 
\[
S_{\lambda}^{\triangle}\varphi:=S_{\lambda}^{1}\varphi-S_{\lambda}^{2}\varphi.
\]
It's easy to see that
\[
\Theta:=\sup_{\Vert\varphi\Vert\leq1,\varphi\in C_{b}(\Rd)}|S_{\lambda}^{\triangle}\varphi|<\infty.
\]

By the definition of the martingale problem, we have that for $\varphi\in C_{b}^{2}(\Rd),$
\begin{equation}
\mathbf{E}_{i}\varphi(X_{t})-\varphi(x_{0})=\mathbf{E}_{i}\int_{0}^{t}\mathcal{A}\varphi(X_{s})\mathrm{d}s,\ i=1,2.\label{eq 1: proof of main}
\end{equation}
It follows from (\ref{eq 1: proof of main}) and Fubini's theorem
that 
\begin{align*}
\mathbf{E}_{i}\int_{0}^{\infty}e^{-\lambda t}\varphi(X_{t})\mathrm{d}t & =\lambda^{-1}\varphi(x_{0})+\mathbf{E}_{i}\left[\int_{0}^{\infty}e^{-\lambda t}\int_{0}^{t}\mathcal{A}\varphi(X_{s})\mathrm{d}s\mathrm{d}t\right]\\
 & =\lambda^{-1}\varphi(x_{0})+\lambda^{-1}\mathbf{E}_{i}\int_{0}^{\infty}e^{-\lambda t}\mathcal{A}\varphi(X_{t})\mathrm{d}t,
\end{align*}
or
\[
\varphi(x_{0})=S_{\lambda}^{i}(\lambda\varphi-\mathcal{A}\varphi)\ ,\ i=1,2.
\]
So

\begin{equation}
S_{\lambda}^{\triangle}(\lambda\varphi-\mathcal{A}\varphi)=0,\quad\varphi\in C_{b}^{2}(\Rd).\label{eq: S^delta (lambda - A)=00003D 0}
\end{equation}

Let $g$ be a $C^{2}$ function with compact support and let
\[
g_{\varepsilon}(x):=\int_{\varepsilon}^{\infty}\int_{\Rd}e^{-\lambda t}q(t,x,y)g(y)\mathrm{d}y\mathrm{d}t,\quad x\in\Rd,
\]
where $q(t,x,y)=f_{t}^{y}(y-x)$ is defined in (\ref{eq: defi q}).
\textcolor{black}{By (\ref{bound:gradientg_t-1}), we see that}\textcolor{red}{{}
${\color{black}g_{\varepsilon}\in C_{b}^{2}(\Rd)}$}. We have
\begin{align*}
(\lambda\text{\textminus}\mathcal{A})g_{\varepsilon} & (x)=(\lambda\text{\textminus}\mathcal{A})\left(\int_{\varepsilon}^{\infty}\int_{\Rd}e^{-\lambda t}q(t,x,y)g(y)\mathrm{d}y\mathrm{d}t\right)\\
 & =\int_{\varepsilon}^{\infty}\int_{\Rd}e^{-\lambda t}\left[(\lambda\text{\textminus}\mathcal{A})q(t,\cdot,y)\right](x)g(y)\mathrm{d}y\mathrm{d}t\\
 & =\int_{\varepsilon}^{\infty}\int_{\Rd}e^{-\lambda t}\left[(\lambda\text{\textminus}\mathcal{A}^{y})q(t,\cdot,y)\right](x)g(y)\mathrm{d}y\mathrm{d}t\\
 & \qquad+\int_{\varepsilon}^{\infty}\int_{\Rd}e^{-\lambda t}\left[(\mathcal{A}^{y}-\mathcal{A})q(t,\cdot,y)\right](x)g(y)\mathrm{d}y\mathrm{d}t\\
 & \quad=:I_{\varepsilon}(x)+J_{\varepsilon}(x).
\end{align*}
Since $\partial_{t}q(t,x,y)=\partial_{t}\left(f_{t}^{y}(y-x)\right)=\mathcal{A}^{y}\left(q(t,\cdot,y)\right)(x)$
for $t>0$ and $x,y\in\Rd$, by Fubini and integration by parts, we
get 
\begin{equation}
I_{\varepsilon}(x)=\int_{\Rd}\left(\int_{\varepsilon}^{\infty}e^{-\lambda t}\left[(\lambda\text{\textminus}\mathcal{A}^{y})q(t,\cdot,y)\right](x)\mathrm{d}t\right)g(y)\mathrm{d}y=\int_{\Rd}e^{-\lambda\varepsilon}q(\varepsilon,x,y)g(y)\mathrm{d}y.\label{eq: 1.2, proof main}
\end{equation}

We now show that $I_{\varepsilon}(x)$ goes to $g(x)$ as $\varepsilon\to0$.
Let $k\in\BN$. We can choose $\delta\in(0,1)$ small enough so that
\begin{equation}
\sup_{|x-y|\le\delta}|g(x)-g(y)|\le\frac{1}{k}.\label{eq: 1.5, proof main}
\end{equation}
For $0<t\le1$ and $z\in\Rd$, we have 
\begin{align*}
\int_{|y-x|>\delta}f_{t}^{z}(y-x)\mathrm{d}y & \overset{(\ref{bound:gradientg_t-1})}{\le}c_{1}\int_{\delta<|y-x|\le1}\frac{t}{|y-x|^{d+\overline{\alpha}}}\mathrm{d}y+c_{1}\int_{|y-x|>1}\frac{t}{|y-x|^{d+\underline{\alpha}}}\mathrm{d}y\\
 & \ \le c_{2}t\left(\int_{\delta}^{1}r^{-1-\overline{\alpha}}\mathrm{d}r+\int_{1}^{\infty}r^{-1-\underline{\alpha}}\mathrm{d}r\right).
\end{align*}
It follows that there exists $t_{0}>0$ such that 
\begin{equation}
\int_{|y-x|>\delta}f_{t}^{z}(y-x)\mathrm{d}y<\frac{1}{k},\quad\mbox{for all \ensuremath{t\le t_{0}}\ and \ensuremath{x,z\in\Rd.} }\label{eq2: proof main}
\end{equation}
So, for $\varepsilon<t_{0}$,
\begin{align}
 & \left|\int_{\Rd}q(\varepsilon,x,y)g(y)\mathrm{d}y-g(x)\right|\nonumber \\
 & \quad=\left|\int_{\Rd}f_{\varepsilon}^{y}(y-x)g(y)\mathrm{d}y-g(x)\int_{\Rd}f_{\varepsilon}^{x}(y-x)\mathrm{d}y\right|\nonumber \\
 & \quad\overset{(\ref{eq2: proof main})}{\le}\frac{2\|g\|}{k}+\left|\int_{|y-x|\le\delta}f_{\varepsilon}^{y}(y-x)\left[g(y)-g(x)\right]\mathrm{d}y\right|\nonumber \\
 & \qquad+\left|\int_{|y-x|\le\delta}f_{\varepsilon}^{y}(y-x)g(x)\mathrm{d}y-g(x)\int_{|y-x|\le\delta}f_{\varepsilon}^{x}(y-x)\mathrm{d}y\right|\nonumber \\
 & \quad\overset{(\ref{eq: 1.5, proof main})}{\le}\frac{2\|g\|}{k}+c_{3}k^{-1}\int_{\Rd}\varrho_{\alpha(y)}^{0,\alpha(y)}(\varepsilon,y-x)\mathrm{d}y\nonumber \\
 & \qquad+\|g\|\int_{|y-x|\le1}\left|f_{\varepsilon}^{y}(y-x)-f_{\varepsilon}^{x}(y-x)\right|\mathrm{d}y,\label{eq:3, proof main}
\end{align}
where $c_{3}>0$ is a constant depending only on $d,\underline{\alpha},\overline{\alpha},\kappa_{1},\kappa_{2}$.
By Proposition \ref{lem: integral f_t^y - f_t^x}, the term in (\ref{eq:3, proof main})
converges to $0$ as $\varepsilon\to0$. Therefore,
\begin{align*}
 & \limsup_{\varepsilon\to0}\left|\int_{\Rd}q(\varepsilon,x,y)g(y)\mathrm{d}y-g(x)\right|\\
 & \quad\le\frac{2\|g\|}{k}+c_{3}k^{-1}\limsup_{\varepsilon\to0}\int_{\Rd}\varrho_{\alpha(y)}^{0,\alpha(y)}(\varepsilon,y-x)\mathrm{d}y\overset{(\ref{esti1:rho})}{\le}\frac{2\|g\|}{k}+c_{4}k^{-1}.
\end{align*}
Here $c_{4}>0$ is also a constant depending only on $d,\underline{\alpha},\overline{\alpha},\kappa_{1},\kappa_{2}$.
Letting $k\to\infty$ yields
\[
\lim_{\varepsilon\to0}\int_{\Rd}q(\varepsilon,x,y)g(y)\mathrm{d}y=g(x),\quad x\in\Rd.
\]
In view of (\ref{eq: 1.2, proof main}), it is clear that $\lim_{\varepsilon\to0}I_{\varepsilon}(x)=g(x)$. 

According to Proposition \ref{prop: <12}, there exists $\lambda_{0}>0$
such that for all $\lambda\geq\lambda_{0}$ and $x\in\Rd$,
\[
|J_{\varepsilon}(x)|\le\Vert g\Vert\sup_{x\in\Rd}\int_{0}^{\infty}\int_{\Rd}e^{-\lambda t}|F(t,x,y)\mathrm{|d}y\mathrm{d}t\leq\frac{1}{2}\Vert g\Vert.
\]
Moreover, by (\ref{bound:gradientg_t-1}) and the dominated convergence
theorem, we can easily verify that $J_{\varepsilon}(x)$ is continuous
in $x$. So $J_{\varepsilon}\in C_{b}(\Rd)$ if $\lambda\geq\lambda_{0}$.

Let $\lambda\ge\lambda_{0}$. Since $S_{\lambda}^{\triangle}(\lambda-\mathcal{A})g_{\varepsilon}=0$
by (\ref{eq: S^delta (lambda - A)=00003D 0}), we have $|S_{\lambda}^{\triangle}I_{\varepsilon}|=|S_{\lambda}^{\triangle}J_{\varepsilon}|$.
Letting $\varepsilon\rightarrow0$ and applying the dominated convergence
theorem, we obtain
\begin{equation}
|S_{\lambda}^{\triangle}g|=\lim_{\varepsilon\rightarrow0}|S_{\lambda}^{\triangle}I_{\varepsilon}|=\lim_{\varepsilon\rightarrow0}|S_{\lambda}^{\triangle}J_{\varepsilon}|\leq\Theta\limsup_{\varepsilon\rightarrow0}\Vert J_{\varepsilon}\Vert\leq\frac{1}{2}\Theta\Vert g\Vert.\label{eq: bound for difference of resolvents}
\end{equation}

We now proceed to extend the above inequality to all $g\in C_{b}(\Rd)$.
First assume $g\in C_{b}(\Rd)$ and $g$ has compact support. If $\left\{ \phi_{\epsilon}\right\} $
is a mollifier sequence, then $g_{\epsilon}:=g\ast\phi_{\epsilon}\in C_{c}^{\infty}(\Rd)$
and thus 
\[
|S_{\lambda}^{\triangle}g_{\epsilon}|\leq\frac{1}{2}\Theta\Vert g_{\epsilon}\Vert\le\frac{1}{2}\Theta\Vert g\Vert.
\]
Passing to the limit as $\epsilon\to0$, we obtain \eqref{eq: bound for difference of resolvents}
by the dominated convergence theorem. Now, take a general $g\in C_{b}(\Rd)$
and let $\varphi\in C_{c}^{\infty}(\Rd)$ be such that $\mathbf{1}_{\left\{ |x|\le1\right\} }\le\varphi\le\mathbf{1}_{\left\{ |x|\le2\right\} }$.
Define $(\varphi_{j})_{j\ge1}\subset C_{c}^{\infty}(\Rd)$ by $\varphi_{j}(y):=\varphi(y/j)$.
By the dominated convergence theorem and the result we just obtained
in the previous step, we get 
\[
|S_{\lambda}^{\triangle}g|=\lim_{j\to\infty}|S_{\lambda}^{\triangle}\left(\varphi_{j}g\right)|\leq\frac{1}{2}\Theta\limsup_{j\to\infty}\Vert\varphi_{j}g\Vert\le\frac{1}{2}\Theta\Vert g\Vert.
\]
So \eqref{eq: bound for difference of resolvents} holds for all $g\in C_{b}(\Rd)$,
which implies
\[
\Theta=\sup_{\Vert g\Vert\leq1,g\in C_{b}(\Rd)}|S_{\lambda}^{\triangle}g|\le\frac{1}{2}\Theta.
\]
Since $\Theta<\infty$, it follows that $\Theta=0,$ or equivalently,
\begin{equation}
\mathrm{\mathbf{E}}_{1}\int_{0}^{\infty}e^{-\lambda t}f(X_{t})\mathrm{d}t=\mathrm{\mathbf{E}}_{2}\int_{0}^{\infty}e^{-\lambda t}f(X_{t})\mathrm{d}t,\quad f\in C_{b}(\Rd).\label{eq: laplace tran, proof main}
\end{equation}
Note that (\ref{eq: laplace tran, proof main}) holds for all $\lambda\ge\lambda_{0}$.
By the uniqueness of the Laplace transform and the right continuity
of $t\mapsto\mathrm{\mathbf{E}}_{i}f(X_{t})$, we obtain $\mathrm{\mathbf{E}}_{1}f(X_{t})=\mathrm{\mathbf{E}}_{2}f(X_{t})$
for all $t\ge0$ and $f\in C_{b}(\Rd)$. This says that the one-dimensional
distributions of any two solutions to the martingale problem agree.
As well-known, this already implies uniqueness for the martingale
problem (see \cite{MR2190038} for details). The proposition is proved. 
\end{proof}

\section{Proof of Theorem \ref{thm: main}}

In this section we will prove Theorem \ref{thm: main}. The main task
is to remove the condition $\overline{\alpha}<2\underline{\alpha}$
that we assumed in the last section. This can be achieved by the standard
localization procedure.

Due to Assumption \ref{assu: main}(d), there exists a constant $0<\delta<1$
such that 
\begin{equation}
|\alpha(x)-\alpha(y)|\le5^{-1}\underline{\alpha},\quad\mbox{whenever \ensuremath{|x-y|\le\delta.}}\label{eq: chioce of delta}
\end{equation}
Let $B_{\delta}(x):=\{y:|y-x|<\delta\}$ and $\overline{B_{\delta}(x)}:=\{y:|y-x|\le\delta\}$.
Note that \eqref{eq: chioce of delta} implies that for each $x\in\Rd$,
\begin{equation}
\sup_{y\in\overline{B_{\delta}(x)}}\alpha(y)\le\alpha(x)+5^{-1}\underline{\alpha}\le\frac{3}{2}\left(\alpha(x)-5^{-1}\underline{\alpha}\right)\le\frac{3}{2}\inf_{y\in\overline{B_{\delta}(x)}}\alpha(y).\label{eq: local extra cond.}
\end{equation}

We first establish the local uniqueness.
\begin{lem}
\label{lem:Let-.-Suppose last}Let $x\in\Rd$. Suppose $\mathbf{P}^{x}$
and $\mathbf{Q}^{x}$ are solutions to the martingale problem for
$\mathcal{A}$ starting from $x$. Define $\tau_{1}:=\inf\left\{ t\ge0:\ X_{t}\notin B_{\delta}(X_{0})\right\} $,
where $\delta$ is as in \eqref{eq: chioce of delta}. Then 
\begin{equation}
\mathbf{P}^{x}(B)=\mathbf{Q}^{x}(B),\quad\forall B\in\mathcal{F}_{\tau_{1}}.\label{eq: to be proved local lemma}
\end{equation}
\end{lem}

\begin{proof}
Define a map $T:\overline{B_{2\delta}(x)}\setminus B_{\delta}(x)\to\overline{B_{\delta}(x)}$
by 
\[
T(y)=x+\frac{\left(y-x\right)(2\delta-|y-x|)}{|y-x|},\quad y\in\overline{B_{2\delta}(x)}\setminus B_{\delta}(x).
\]
Not to be precise, $T$ is the mirror image map from $\overline{B_{2\delta}(x)}\setminus B_{\delta}(x)$
to $\overline{B_{\delta}(x)}$ with respect to the sphere surface
$\{z:|z-x|=\delta\}$. It is easy to see that $T$ is Lipschitz continuous,
namely, there exists a\textcolor{red}{{} }constant $c_{1}>1$ such that
\begin{equation}
|T(y)-T(y')|\le c_{1}|y-y'|,\quad y,y'\in\overline{B_{2\delta}(x)}\setminus B_{\delta}(x).\label{eq: prop 1 for T}
\end{equation}
Note also that if $z\in B_{\delta}(x)$, then 
\begin{equation}
|z-T(y)|\le|z-y|,\quad\forall y\in\overline{B_{2\delta}(x)}\setminus B_{\delta}(x).\label{eq: prop 2 for T}
\end{equation}

Based on $\mathcal{A}$, we define a new operator $\tilde{\mathcal{A}}$
by modifying the values of $\alpha(y)$ for $y\notin\overline{B_{\delta}(x)}$,
namely, 
\[
\mathcal{\tilde{A}}f(y)=\int_{\mathbb{R}^{d}\backslash\{0\}}\left[f(y+h)-f(y)-\mathbf{1}_{\{|h|\le1\}}h\cdot\nabla f(y)\right]\frac{n(y,h)}{|h|^{d+\tilde{\alpha}(y)}}\mathrm{d}h,\quad y\in\Rd,
\]
where 
\[
\tilde{\alpha}(y):=\begin{cases}
\alpha(y), & y\in\overline{B_{\delta}(x)},\\
\alpha\left(T(y)\right), & y\in\overline{B_{2\delta}(x)}\\
\alpha(x), & y\notin\overline{B_{2\delta}(x)}.
\end{cases}\setminus\overline{B_{\delta}(x)},
\]
It follows from \eqref{eq: local extra cond.} that \emph{
\[
\sup_{y\in\Rd}\tilde{\alpha}(y)\le\frac{3}{2}\inf_{y\in\Rd}\tilde{\alpha}(y).
\]
}

We now verify that $\mathcal{\tilde{A}}$ satisfies Assumption\emph{
}\ref{assu: main}. In fact, we only need to check that $\tilde{\beta}(r)=o(|\ln r|^{-1})$
as $r\to0$ and 
\[
\int_{0}^{1}r^{-1}|\ln r|\tilde{\beta}(r)\mathrm{d}r<\infty,
\]
where $\tilde{\beta}(r):=\sup_{|x-y|\le r}|\tilde{\alpha}(x)-\tilde{\alpha}(y)|$.
To verify these two conditions, it suffices to show
\begin{equation}
\tilde{\beta}(r)\le\beta(c_{1}r),\quad\forall r>0.\label{eq: condi for beta tilde}
\end{equation}
For $y,y'\in\overline{B_{2\delta}(x)}\setminus B_{\delta}(x)$, we
have 
\[
|\tilde{\alpha}(y)-\tilde{\alpha}(y')|=|\alpha(T(y))-\alpha(T(y'))|\le\beta(|T(y)-T(y')|)\overset{(\ref{eq: prop 1 for T})}{\le}\beta(c_{1}|y-y'|).
\]
For $y\in B_{\delta}(x)$ and $y'\in\overline{B_{2\delta}(x)}\setminus B_{\delta}(x)$,
we have 
\[
|\tilde{\alpha}(y)-\tilde{\alpha}(y')|=|\alpha(y)-\alpha(T(y'))|\overset{(\ref{eq: prop 2 for T})}{\le}\beta(|y-y'|)\le\beta(c_{1}|y-y'|).
\]
The case for $y\in\overline{B_{2\delta}(x)}$ and $y'\notin\overline{B_{2\delta}(x)}$
is similar. Altogether, we see that (\ref{eq: condi for beta tilde})
is true. So Assumption\emph{ }\ref{assu: main} holds true for $\mathcal{\tilde{A}}$.

In view of Proposition \ref{prop:Let--be main}, the martingale problem
for $\tilde{\mathcal{A}}$ is well-posed. Let $\mathbf{L}^{y}$ be
the solution to the martingale problem for $\mathcal{\tilde{A}}$
starting from $y\in\Rd$. According to \cite[Chap.4,~Theorem 4.6]{MR838085}
(see also \cite[Exercise 6.7.4]{MR2190038}), the mapping $y\mapsto\mathbf{L}^{y}$
is measurable. Now define $\mathbf{\tilde{P}}^{x}$ by 
\[
\mathbf{\tilde{P}}^{x}\left(B\cap\left(C\circ\theta_{\tau_{1}}\right)\right)=\mathbf{E}_{\mathbf{P}^{x}}\left[\mathbf{L}^{X_{\tau_{1}}}(C);B\right],\quad B\in\mathcal{F}_{\tau_{1}},C\in\mathcal{D},
\]
where $\theta_{t}$ are the usual shift operators on $D=D\big([0,\infty);\Rd\big)$.
Let $\mathbf{\tilde{Q}}^{x}$ be defined in the same way. Then it
is routine to check that $\mathbf{\tilde{P}}^{x}$ and $\mathbf{\tilde{Q}}^{x}$
are solutions to the martingale problem for $\mathcal{\tilde{A}}$
starting from $x$. So $\mathbf{\tilde{P}}^{x}=\mathbf{\tilde{Q}}^{x}$
by Proposition \ref{prop:Let--be main}. By the definition of $\mathbf{\tilde{P}}^{x}$
and $\mathbf{\tilde{Q}}^{x}$, we obtain \eqref{eq: to be proved local lemma}.
The lemma is proved.
\end{proof}
Finally, we give the proof of our main result.

\emph{Proof of Theorem} \ref{thm: main}. Let $x\in\Rd$. Suppose
$\mathbf{P}^{x}$ and $\mathbf{Q}^{x}$ are solutions to the martingale
problem for $\mathcal{A}$ starting from $x\in\Rd$.

Let $\delta$ and $\tau_{1}$ be as in Lemma \ref{lem:Let-.-Suppose last}.
Define inductively 
\[
\tau_{i+1}:=\{t\ge\tau_{i}:X_{t}\notin B_{\delta}(X_{\tau_{i}})\}.
\]
In view of Lemma \ref{lem:Let-.-Suppose last}, we can use standard
argument (see, for instance, \cite[Section 6.3,~Theorem 3.4]{MR1483890})
to conclude that 
\[
\mathbf{P}^{x}(B)=\mathbf{Q}^{x}(B),\quad\forall B\in\mathcal{F}_{\tau_{n}},n\in\mathbb{N}.
\]
To see $\mathbf{P}^{x}=\mathbf{Q}^{x}$, it remains to show that $\mathbf{P}^{x}(\tau_{n}\to\infty)=\mathbf{Q}^{x}(\tau_{n}\to\infty)=1$. 

Let $\sigma_{r}:=\inf\left\{ t\ge0:X_{t}\notin B_{r}(X_{0})\right\} $.
Keeping in mind the symmetry property $n(x,h)=n(x,-h)$, we can repeat
the proof of \cite[Proposition 3.1]{MR2095633} to find a constant
$c_{1}>0$ such that for all $y\in\Rd$ and $0<r<1$, 
\[
\mathbf{P}^{y}(\sigma_{r}\le c_{1}r^{\overline{\alpha}})\le\frac{1}{2},
\]
where $\mathbf{P}^{y}$ is any solution to the martingale problem
for $\mathcal{A}$ starting from $y$. In particular, we have 
\[
\mathbf{P}^{y}(\tau_{1}\le\epsilon)\le\frac{1}{2},\quad\forall y\in\Rd,
\]
where $\epsilon>0$ is some constant not depending on $y$. As shown
in the proof of \cite[Corollary 4.4]{MR1964949}, this implies, for
some constant $\gamma$, 
\[
\mathbf{E}_{\mathbf{P}^{y}}(e^{-\tau_{1}})\le\gamma<1,\quad\forall y\in\Rd.
\]
Using the strong Markov property, we get
\[
\mathbf{E}_{\mathbf{P}^{x}}(e^{-\tau_{n}})\le\gamma^{n}\to0,\quad\mbox{as \ensuremath{n\to\infty}, }
\]
which implies $\tau_{n}\to\infty$ $\mathbf{P}^{x}$-a.s. The same
statement holds also for $\mathbf{Q}^{x}$. So $\mathbf{P}^{x}=\mathbf{Q}^{x}$.
The theorem is proved. \qed 

\subsection*{Acknowledgments}

The author would like to thank the anonymous referees for their valuable
comments and suggestions. The author is supported by the STU Scientific
Research Foundation for Talents (No. NTF18023) and NSFC (No. 11861029).

\bibliographystyle{amsplain}
\addcontentsline{toc}{section}{\refname}\bibliography{PJ_Libarary}

\end{document}